\documentclass[reqno,12pt]{amsart} 
\usepackage{esint,mathrsfs,amsmath,amssymb,amsfonts}
\usepackage{verbatim} %, showkeys
\usepackage{color}
\usepackage[hyperfootnotes=false]{hyperref}
\hypersetup{
  colorlinks,
  citecolor=blue,
  linkcolor=red,
  urlcolor=blue}
\usepackage[left=1.2in, right=1.2in, bottom=1.5in]{geometry}

\newtheorem{theorem}{Theorem}[section]
\newtheorem{lemma}[theorem]{Lemma}
\newtheorem{proposition}[theorem]{Proposition}
\newtheorem{corollary}[theorem]{Corollary}
\theoremstyle{definition}

\theoremstyle{remark}
\newtheorem{remark}[theorem]{Remark}
\numberwithin{equation}{section}

\newcommand{\norm}[1]{\lVert#1\rVert}

\newcommand{\cL}{\mathcal{L}}

\newcommand{\C}{\mathbb{C}}
\newcommand{\R}{\mathbb{R}}

\newcommand{\e}{\varepsilon}

\newcommand{\dist}{{\rm dist}}

\begin{document}

\title[]{Quantitative unique continuation for Neumann problem in planar $C^{1,\alpha}$ domains}
\author{Yingying Cai}
\address{Yingying Cai: Departament de Matemàtiques, Universitat Autònoma de Barcelona, Barcelona, Spain.}
\email{yingying.cai@uab.cat}

\author{Jiuyi Zhu}
\address{Jiuyi Zhu: Department of Mathematics, Louisiana State University, Baton Rouge, LA 70810.}
\email{zhu@math.lsu.edu}

\author{Jinping Zhuge}
\address{Jinping Zhuge: Morningside Center of Mathematics, Academy of Mathematics and systems science,
Chinese Academy of Sciences, Beijing 100190, China.}
\email{jpzhuge@amss.ac.cn}

\subjclass[2010]{35J25, 35B60, 35A02}

\begin{abstract}
    In this paper, we study the quantitative unique continuation property of the second-order elliptic operators under the vanishing Neumann boundary condition over $C^{1,\alpha}$ or convex domains in two dimensions. 
    We establish the optimal estimates of the number of critical points, doubling index and the total length of level curves. The key idea is to reduce the Neumann problem to the Dirichlet problem, which has been understood better, by a classical duality between an $A$-harmonic function and its stream function.

    \noindent \textbf{Keywords:} Unique continuation, Neumann problem, stream function, doubling index.
\end{abstract}
\maketitle

\tableofcontents

\section{Introduction}

\subsection{Motivation}

The unique continuation properties (UCP) for elliptic equations is a fundamental topic in classical analysis, extensively explored over the past decades. Recent work has increasingly focused on the quantitative UCP, including the explicit bound of vanishing order, doubling index, three-ball (or three-sphere) inequality, propagation of smallness, and quantitative estimate of nodal sets and critical/singular sets, etc, particularly for the second-order elliptic equations with Lipschitz coefficients or the Laplacian eigenfunctions on Riemannian manifolds; see e.g. \cite{GL86,DF88, L91,NV17,L18a,L18b,LLS24} and reference therein.
The Lipschitz regularity of the coefficients is a threshold for the UCP in three dimensions or higher, which may fail for general $C^{1-\e}$ coefficients for any $\e > 0$ \cite{P63,Man98}. 

On the other hand, the analysis becomes more complex when studying the quantitative UCP near a nonsmooth boundary under a certain boundary condition (such as Dirichlet or Neumann boundary conditions). The quantitative (or qualitative) UCP may even fail for less regular boundaries; see counterexamples in \cite{BW90,Hir18,KZ22}. We provide a brief review on the developments in this direction. Under the vanishing Dirichlet boundary condition, it is more or less well-known that the almost monotonicity of frequency function (a classical tool to study the UCP) centered on a boundary is valid if the boundary satisfies the so-called $A$-starshape condition (i.e., $n\cdot A(x-x_0) \ge 0$; see \cite{ZZ23} for a precise definition). This condition can be achieved for free by a simple change of variables if the boundary is $C^{1,1}$ (also see \cite{AEK95} for results in convex domains). Then $C^{1,\alpha}$ (or $C^{1,{\rm Dini}}$) boundaries for the Dirichlet problem were studied in the pioneering work by Adolfsson-Escauriaza \cite{AE97} (also see \cite{KN98}) 
by a refined change of variables that recovers $A$-starshape condition. In fact, it was shown that the doubling index and vanishing order are uniformly bounded on $C^{1,\alpha}$ boundary with a vanishing Dirichlet boundary condition. For less regular domains, such as Lipschitz domains, the boundedness of doubling index is still unknown. In  \cite[Theorem 2.3 and its Remark]{L91}, Lin asked in a Lipschitz domain if a harmonic function is trivial if it vanishes in an open set of the boundary and its normal derivative vanishes in a subset of positive surface measure. Recently, for Laplace operator, Tolsa \cite{T21} showed that on $C^1$ boundaries the doubling index is bounded a.e. with respect to the surface measure and thus confirmed Lin's conjecture for $C^1$ boundaries. Then Logunov-Malinnikova-Nadirashvili-Nazarov \cite{LMNN21} established the optimal estimate of nodal sets with $C^1$ boundaries (also see \cite{G22} for general elliptic operators). These results were finally extended in \cite{ZZ23,Cai25} to a unified class of quasiconvex Lipschitz domains that includes both $C^1$ and convex domains as proper subsets (also see the results in convex domains \cite{AEK95} or polytopes \cite{CZ24}). We also refer to \cite{KZ21,KZ22,KZ25} for the quantitative estimate of critical/singular sets in $C^{1,{\rm Dini}}$ domains and to \cite{M19} for quantitative results in convex domains. 

For the Neumann boundary value problem, the study of UCP has been stuck on $C^{1,1}$ boundaries since the work \cite{AE97}. 
In order to prove the monotonicity of frequency function or Carleman estimate under the vanishing Neumann boundary condition, we need the boundary to be exactly ``$A$-flat'' (i.e., $n\cdot A(x-x_0) = 0$), which is much more restrictive than being $A$-starshape. For $C^{1,1}$ boundaries, Adolfsson-Escauriaza \cite{AE97} constructed a special transformation to ``$A$-flatten'' the boundary such that the monotonicity formula of frequency function is available. Unfortunately, this particular transformation crucially depends on the $C^{1,1}$ regularity of the boundary.  For boundaries less regular than $C^{1,1}$, a transformation $A$-flattening the boundary will completely destroy the Lipschitz regularity of coefficients and hence at present there is no known monotonicity formula or Carleman estimate available. We also mention some recent literature on the UCP for Neumann/Robin boundary value problems \cite{Zhu23,BZ23,LW23}, and all these works assume at least $C^{1,1}$ boundaries.

In this paper, we will establish the quantitative UCP for elliptic equations subject to the vanishing Neumann boundary condition in the planar $C^{1,\alpha}$ domains. To the best of our knowledge, this should be the first quantitative study of UCP for Neumann problem with a boundary less regular than $C^{1,1}$. Our main results are restricted in two dimensions as we will employ special techniques available only in two dimensions. The questions on the quantitative UCP in higher dimensions for Neumann problems in less regular domains deserve more explorations.

\subsection{Statement of main results}
Consider the second-order elliptic operator defined by $\mathcal{L}=-\nabla \cdot A\nabla$ with $A$ being a positive-definite matrix in a planar domain $\Omega \subset \R^2$. A function $u\in H^1(\Omega)$ is called $A$-harmonic in $\Omega$ if $\mathcal{L}(u)=0$ in the weak sense in $\Omega$. We say $u$ satisfies the vanishing Neumann boundary condition on $\Gamma \subset \partial \Omega$ if $\frac{\partial u}{\partial \nu} := n\cdot A\nabla u = 0$ in the weak sense on $\Gamma$, where $n$ is the outer normal to $\partial \Omega$.
In this paper, we assume that the coefficient matrix $A$ satisfies the following standard assumptions:

\begin{itemize}
    \item Ellipticity: there exists $\Lambda\ge 1$ such that 
    \begin{equation}\label{ellipticity}
        \Lambda^{-1}\mathbin{I}\le A(x)\le \Lambda\mathbin{I} \qquad \text{ for any } x\in {\Omega}.
    \end{equation}
    \item Symmetry:
    \begin{equation}\label{symmetry}
        A^T = A.
    \end{equation}

    \item $C^\alpha$-H\"{o}lder continuous, i.e., for some $\alpha \in (0,1)$ and $\gamma \ge 0$
    \begin{equation}\label{A-Holder}
    |A(x) - A(y)| \le \gamma |x-y|^\alpha \qquad \text{for every } x,y \in \Omega.
    \end{equation}
    
\end{itemize}

It is well-known that the quasiconformal mapping is a standard tool to study the elliptic equation in two dimensions. Actually, even with bounded measurable coefficients, the strong unique continuation property (SUCP) holds at interior points and boundary points (reduced to the interior case); see e.g. \cite{Sch98}. The problem is that, without any additional regularity or structure, the complex analysis tools alone do not give useful quantitative estimates, due to the degenerate distortion of the quasiconformal mapping (see \cite{AIM09} and Subsection \ref{subsec:quasiconformal} for more details). Therefore, the generality of our elliptic operator $\cL$ prevents us from using complex analysis tools alone to circumvent the essential difficulties to obtain the (optimal) quantitative estimates.

Our first main result is a quantitative upper bound of the number of critical points, in terms of the doubling index that characterizes the growth rate of a solution.

\begin{theorem}\label{thm.criticalpts}
    Let $A$ satisfy \eqref{ellipticity} - \eqref{A-Holder} in a bounded $C^{1,\alpha}$ planar domain $\Omega$ with some $\alpha \in (0,1)$. Let $0\in \partial \Omega$ and $u$ be $A$-harmonic (and nontrivial) in $B_R \cap \Omega$ subject to the vanishing Neumann boundary condition on $B_R \cap \partial \Omega$, where $B_R := B_R(0)$. Then there exist constants $r_0, C$ and $\theta \in (0,1/2)$, depending only on $A$ and $\Omega$, such that if $R\in (0,r_0)$, then
    $u$ has a finite number of critical points and
    \begin{equation}\label{est.CuSharp}
        \# (\mathcal{C}(u)\cap B_{\theta R}) \le C\log \frac{\| \nabla u \|_{L^2(B_R \cap \Omega)}}{ \| \nabla u \|_{L^2(B_{\theta R} \cap \Omega ) }  },
    \end{equation}
    where $\#$ denotes the counting measure and $\mathcal{C}(u) = \{ x\in \overline{\Omega}: \nabla u(x) = 0 \}$. 
   \end{theorem}

The above estimate \eqref{est.CuSharp} is optimal, in view of the classical interior results \cite{Han07, HLbook,Zhu23b} in two dimensions, where at least Lipschitz regularity of the leading coefficients is required. The $C^{1,\alpha}$ boundary may be relaxed to $C^{1, {\rm Dini}}$ by checking the proof (also see \cite{KZ25} for the Dirichlet problem of Laplace operator). However, the number of boundary critical points may be infinite if the boundary is $C^1$ or convex (see counterexample in \cite{KZ22} for Dirichlet problem), in contrast to the following results on the estimate of doubling index and the length of level curves. The failure for the convex domain is not surprising, as each convex cone will generate a critical point at the tip and therefore the total number of these boundary critical points may accumulate to infinity. Nevertheless, the Hausdorff dimension of boundary critical set in this case is still zero; see \cite{M19}. Also see Remark \ref{rmk.intCP} for the interior result off the boundary.

Our second result is the optimal upper bound of doubling index in a neighborhood of the boundary.
\begin{theorem}\label{thm.Doubling}
    Let $\Omega$ be either a bounded $C^{1,\alpha}$ or convex domain in $\R^2$. 
    Let $A$ satisfy \eqref{ellipticity}, \eqref{symmetry} and
    \begin{itemize}
        \item Lipschitz continuity: there exists $\gamma \ge 0$ such that  
    \begin{equation}\label{Lip}
        | A(x)-A(y) |\le \gamma |x-y| \qquad \text{ for every } x,y \in {\Omega}.
    \end{equation}
    \end{itemize}
    Let $u$ be as before. Then there exist $r_0, C$ and $\theta \in (0,1/2)$ such that if $R\in (0,r_0)$,
    \begin{equation}\label{est.UnifDoubling}
        \sup_{\substack{x\in B_{\theta R} \cap \Omega\\ 0<r< \theta R} } N_u(x,r) \le C \log \frac{\| u \|_{L^2(B_{R} \cap \Omega)}}{\| u \|_{L^2(B_{\theta R} \cap \Omega)}},
    \end{equation}
    where the doubling index $N_u(x,r)$ is defined in Subsection \ref{subsec:DefDoubling}.
\end{theorem}

The uniform boundedness of doubling index on the boundary has many consequences. In the following, for simplicity, we briefly list some results implied by Theorem \ref{thm.Doubling}, which will be proved rigorously (or implied directly) throughout the paper.
\begin{itemize}
    \item Boundedness of vanishing order on the boundary (well-known).

    \item Estimate of the length of level curves (or level sets); see Theorem \ref{thm.LevelSet} below.

    \item Estimate of the number of level points on the boundary (or equivalently the number of level curves touching the boundary) and Cauchy uniqueness for Neumann problem; see Theorem \ref{thm.levelpoints} and Corollary \ref{Coro.CauchyUnique}.
\end{itemize}

In particular, we establish the optimal estimate of the total length of the level curves. Let $Z_t(u) = \{ x\in \Omega: u(x) =t \}$. The nodal curve of $u$ then is given by $Z(u) = Z_0(u)$. Note that $Z_t(u) = Z_0(u -t)$.
\begin{theorem}\label{thm.LevelSet}
    Under the same assumptions of Theorem \ref{thm.Doubling}, there exists $C \ge 0$ (depending only on $A$ and $\Omega$) such that for any $t\in \R$, 
\begin{equation}\label{est.LevelSet=t}
        \mathcal{H}^1(Z_t(u) \cap B_{\theta R}) \le CR \log \frac{\| u - t\|_{L^2(B_R \cap \Omega)}}{ \| u -t \|_{L^2(B_{\theta R} \cap \Omega )}},
    \end{equation}
    where $\mathcal{H}^1$ denotes the one-dimensional Hausdorff measure.
\end{theorem}

The estimates in Theorem \ref{thm.Doubling} and Theorem \ref{thm.LevelSet} are optimal, as evidenced by classical interior examples; see \cite{Ale88,L91}. The Lipschitz continuity assumption \eqref{Lip} is needed to achieve these optimal bounds on the doubling index and level curves. Moreover, Theorem \ref{thm.Doubling} and Theorem \ref{thm.LevelSet} (as well as the following Theorem \ref{thm.WithStructure}) hold also for a broad class of quasiconvex domains\footnote{Precisely, Theorem \ref{thm.Doubling} and Theorem \ref{thm.LevelSet} hold for Lipschitz domains satisfying an exterior $C^{1,\alpha}$ condition, i.e., any boundary point can be touched by a $C^{1,\alpha}$ curve from exterior. }, including both $C^{1,\alpha}$ and convex domains; see \cite{ZZ23}. For the simplicity of presentation, we will not expand the detail in this direction.

Our last theorem shows that the Lipschitz continuity assumption in Theorem \ref{thm.Doubling} and Theorem \ref{thm.LevelSet} can be relaxed to $C^\alpha$-H\"{o}lder continuity, if $A$ satisfies the constant-determinant condition. 
\begin{theorem}\label{thm.WithStructure}
Let $\Omega$ be either $C^{1,\alpha}$ or convex and let $A$ satisfy \eqref{ellipticity} - \eqref{A-Holder}. In addition, assume $\det A$ is constant. Let $u$ be as before.
Then there exist $r_0, C$ and $ \theta \in (0,1/2)$ depending only on $A$ and $\Omega$ such that \eqref{est.UnifDoubling} and \eqref{est.LevelSet=t} hold for any $t\in \R$.
\end{theorem}

The assumption ``$\det(A)$ is constant'' allows us to convert $u$ into a harmonic function via a $C^{1,\alpha}$ diffeomorphism (instead of a $C^\alpha$ quasiconformal mapping).
Furthermore, this extra assumption arises naturally when we consider the Laplace-Beltrami equation in a two-dimensional Riemannian manifold. Recall that the Laplace-Beltrami operator can be locally written as
\begin{equation*}
    \Delta_{\mathrm{g}} u = \frac{1}{\sqrt{g}} \partial_i ( \sqrt{g} g^{ij} \partial_j u ),
\end{equation*}
where $(g^{ij})$ is the inverse of the Riemann metric $(g_{ij})$ and $g = \det (g_{ij})$. Therefore, with $a_{ij} = \sqrt{g} g^{ij}$ and $d = 2$, we have $\det(a_{ij}) = 1$.
Thus, Theorem \ref{thm.WithStructure} applies to this situation. 

\subsection{Main ideas of proof}
The main difficulty for the Neumann problem is that we are unable to directly establish the monotonicity formula for the Almgren-type frequency function if the boundary is merely $C^{1,\alpha}$.
The key idea of the paper (for all four theorems) is to reduce the Neumann boundary value problem to the Dirichlet boundary value problem by a classical duality between an $A$-harmonic function and its stream function  (a generalization of the conjugate harmonic function), for which we have a better understanding from the previous literature. Precisely, for an $A$-harmonic function $u$ in $\Omega\subset \R^2$ , the stream function is given by
\begin{equation}\label{eq.Stream}
    \nabla v=\mathcal{R} A\nabla u,
\end{equation}
where $\mathcal{R}$ is the matrix for a $90^\circ$ counter-clockwise rotation in the plane:
\begin{equation*}
    \mathcal{R} = \begin{pmatrix}
 0 &-1 \\1
  &0
\end{pmatrix}.
\end{equation*}
If $A$ is symmetric, then $v$ satisfies $\nabla\cdot (\widehat{A} \nabla v) = 0$ and $\widehat{A} = (\det A)^{-1} A$ is also symmetric. Moreover, if $u$ satisfies the vanishing Neumann boundary condition on $\partial \Omega \cap B_R$, then    
\begin{equation*}
    n^\perp \cdot \nabla v = n\cdot A \nabla u = 0, \quad \text{on } \partial \Omega \cap B_R,
\end{equation*}
where $n^\perp = (n_2, -n_1)$ and $n = (n_1, n_2)$. This means that the tangential derivative of $v$ vanishes along the boundary $\partial \Omega \cap B_R$. Thus $v$ remains constant on $\partial \Omega \cap B_R$. Without loss of generality, we may assume $v = 0$ on $\partial \Omega \cap B_R$. Therefore, the stream function $v$ is $\widehat{A}$-harmonic and satisfies the vanishing Dirichlet boundary condition. The point is that, by the equation \eqref{eq.Stream}, any quantitative estimates we have for $v$, can be carried over from $\nabla v$ to $\nabla u$ without any essential loss. We emphasize that the duality of boundary conditions between a harmonic function or an incompressible flow and its stream function is classical with many applications in two dimensional PDEs in different contexts (e.g. \cite{BV02,Mat02,ZY25}), particularly in fluid mechanics from which the term ``stream function'' arises (see e.g. \cite{Tem79}). Nevertheless, to the best of our knowledge, this is the first time it is used to study the UCP of the Neumann boundary value problem. 

\textbf{Critical points.} After reducing the Neumann boundary value problem to the Dirichlet boundary value problem, we then apply a geometric approach inspired by \cite{AM94} to estimate the critical points, which is quite different from the ones in recent literature such as \cite{NV17,Zhu23b,KZ25}, etc. The key idea of the geometric approach is the following observation: a (interior or boundary) point $x_0$ is a critical point if and only if it is an intersection of at least two level curves. This fact holds on the boundary for the vanishing Dirichlet boundary condition, but not for the Neumann boundary condition. Fortunately, by \eqref{eq.Stream}, the critical points of $u$ (including those on the boundary) are exactly the critical points of its stream function satisfying the vanishing Dirichlet boundary condition. 

\textbf{Doubling index.} Again, the almost monotonicity and the upper bound of doubling index of $u$ are reduced to the doubling index of $v$ through \eqref{eq.Stream}, while the latter has been studied in \cite{AE97} for $C^{1,\alpha}$ domains and in \cite{ZZ23} for convex (quasiconvex) domains. Since \eqref{eq.Stream} only shows the equivalence between gradients, a few transitions are needed to study the relation of  the doubling index of $v$ to that of $u$.

\textbf{Length of level curves.} The estimate of the total length of level curves follows the approach recently developed in \cite{LMNN21} and \cite{ZZ23}, based on the almost monotonicity and the upper bound of doubling index. Note that estimate of level curves for $u$ cannot be reduced to its stream function $v$, as \eqref{eq.Stream} does not show the connection between their level curves. Instead, we have to adjust the argument, previously developed for the Dirichlet boundary value problem, to the Neumann boundary value problem, with some modifications. Meanwhile, we also need the result on the optimal estimate of interior nodal curves for the two dimensional elliptic equations in \cite{Ale88}.

\textbf{Organizations.} 
In Section \ref{sec:prelim}, we recall the quasiconformal mapping applied to the two dimensional elliptic equations and use it to derive a property on the propagation of doubling index from the boundary to the interior. In Section \ref{sec:criticalpoint}, we prove Theorem \ref{thm.criticalpts} on the number of critical points. In Section \ref{sec:DoublingIndex}, we establish the almost monotonicity and the upper bound of the doubling index, including the proof of Theorem \ref{thm.Doubling}.
In Section \ref{sec:levelset}, we obtain the  length estimate of level curves in Theorem \ref{thm.LevelSet} and its consequence on the Cauchy uniqueness for Neumann problem. In Section \ref{sec:detA=1}, we prove Theorem \ref{thm.WithStructure}.

\textbf{Notations.} The constant $C,c,$ etc are allowed to depend on $A$ and $\Omega$ and may vary from line to line. We write $a \simeq b$ if there exists $C>0$ such that $C^{-1} a \le b \le C a$. We use $B_r(x) = B(x,r)$ to denote the Euclidean ball centered at $x$ with radius $r$, and $B_r$ to denote the ball centered at the origin.

\textbf{Acknowledgments.}
The frist author would like to thank Professor Xavier Tolsa and Professor Liqun Zhang for helpful discussions.
The third author would like to thank Professor Zhongwei Shen for many insightful discussions on this topic. Y. Cai is partially supported by the European Research Council (ERC) under the innovation program (grant agreement 101018680).  J. Zhu is partially supported by  NSF DMS-2154506 and DMS-2453348.
J. Zhuge is partially  supported by NNSF of China (No. 12494541, 12288201, 12471115).

\section{Preliminaries} \label{sec:prelim}

The application of quasiconformal mapping to the second-order elliptic equation in a plane is classical; see e.g. \cite[Chapter 16]{AIM09} and \cite{BJS64,AE08,KZZ22, MZ25}. We provide a brief introduction for clarification. We also provide a technical proposition on the propagation of doubling index from the boundary to the interior.

\subsection{Quasiconformal mapping}
\label{subsec:quasiconformal}
Let $A$ be a $2\times 2$ real-valued matrix given by
\begin{equation*}
    A = \begin{pmatrix}
 a_{11} & a_{12} \\
 a_{21} & a_{22}
\end{pmatrix}.
\end{equation*}
We assume that $A$ is bounded measurable and satisfy the ellipticity condition. We do not need the regularity assumption of $A$ in this subsection. 
Let $u\in H^1(\Omega)$ be an $A$-harmonic function (i.e., $\nabla \cdot (A\nabla u) = 0$) in $\Omega$. Let $v$ be the stream function given by \eqref{eq.Stream}. 

We identify the point $(x_1,x_2) \in \R^2$ as a complex point $z=x_1+ix_2$, and thus write $u(z) = u(x_1, x_2)$. Consider the complex function $f=u+iv$. Then it is straightforward to verify that $f$ satisfies the complex Beltrami equation
\begin{equation}\label{eq-dzf}
    \partial_{\bar{z}}f =\mu \partial_z f+ \nu \overline{\partial_z f} \quad \text{ in }\Omega \subset \C \simeq \R^2,
\end{equation}
where $\partial_{\Bar{z}}f=\frac{1}{2}(\partial_{x_1} f+i\partial_{x_2} f)$, $\partial_z f=\frac{1}{2}(\partial_{x_1} f-i\partial_{x_2} f)$, and
\begin{equation}\label{def.mu+nu}
\left\{ 
\begin{aligned}
    & \mu=\frac{a_{22}-a_{11}-i(a_{12}+a_{21}) }{\det(I + A)}, \\
    & \nu=\frac{1-\text{det}(A) + i(a_{12} - a_{21})}{\det(I + A)}. 
\end{aligned}
\right.
\end{equation}
By the ellipticity condition on $A$, we know that $\mu$ and $\nu$ are bounded measurable and
\begin{equation*}
    |\mu| + |\nu| \le \kappa := \frac{\Lambda-1}{\Lambda+1} < 1.
\end{equation*}
Hence \eqref{eq-dzf} can be written as $\partial_{\bar{z}} f = \tilde{\mu} \partial_z f$ with $|\tilde{\mu}| \le \kappa < 1$. This implies that $f$ is a quasiregular mapping. 

As an important consequence, we can apply the Stoilow Factorization Theorem \cite{AIM09} to write $f$ as
\begin{equation*}
    f(z) = g\circ \chi(z),
\end{equation*}
where $\chi:\Omega \to \Omega'$ is a $\Lambda$-quasiconformal mapping and $g$ is a holomorphic function. Moreover, by taking the real part of $f$, we get
\begin{equation*}
    u = h\circ \chi,
\end{equation*}
where $h$ is a harmonic function on $\Omega' \subset \R^2$.

Particularly, we may apply the above argument to the $A$-harmonic functions on the disk $B_1$. In fact, by \cite{BJS64}, if $u$ is $A$-harmonic in $B_1$, then there exists a $\Lambda$-quasiconformal mapping $\chi: B_1 \to B_1$ with $\chi(0) = 0$ and $\chi(1) = 1$  
such that $u = h\circ \chi$, where $h$ is harmonic in $B_1$. Moreover, for any $z_1, z_2 \in B_1$,
\begin{equation}\label{distortion}
    C^{-1} |z_1 - z_2|^{1/\sigma} \le |\chi(z_1) - \chi(z_2)| \le C|z_1 - z_2|^\sigma,
\end{equation}
where $C$ and $\sigma\in (0,1)$ are positive constants depending only on $\Lambda$. By rescaling, we can obtain similar result for $A$-harmonic functions in $B_R$ for any $R>0$.

Let $w$ be $A$-harmonic in $B_R$ and $w = h\circ \chi$ for some $\Lambda$-quasiconformal mapping $\chi: B_R \to B_R$. Define the quasi-ball 
$$\mathcal{B}_{r}=\{ z\in B_{r}:|\chi(z) |<r \}.$$
By \eqref{distortion}, the quasi-balls $\mathcal{B}_r$ are comparable to the standard Euclidean balls in the sense 
\begin{equation}\label{quasiball}
    \mathcal{B}_{R}=B_{R}\quad \text{and}\quad B_{R(\frac{r}{C_1R})^{1/\sigma}} \subset \mathcal{B}_r\subset B_{R(\frac{C_1r}{R})^\sigma}, \quad \text{for all }\, 0< r<R,
\end{equation}
where $C_1\ge 1$ and $0<\sigma<1$ depend only on $\Lambda$. Then by Hadamard’s three circle theorem, we have the following doubling property in quasi-balls.
\begin{lemma}[{\cite[Proposition 2]{AE08}}]\label{lem.quasiball}
    Let $w$ be a nontrivial $A$-harmonic function in $B_{R}$. Then there exists $C>0$ such that
    \begin{equation}\label{est-quasiball}
     \frac{\norm{w}_{L^\infty(\mathcal{B}_r)}}{\norm{w}_{L^\infty(\mathcal{B}_{r/2})}}\le C\frac{\norm{w}_{L^\infty(\mathcal{B}_{R})}}{\norm{w}_{L^\infty(\mathcal{B}_{R/4})}},\quad \text{ for all }  \; 0<r\le R. 
     \end{equation}
  \end{lemma}

As a corollary, by the distortion relation \eqref{quasiball}, we obtain the doubling property in Euclidean balls.
\begin{corollary}\label{coro-interior-log}
    Let $u$ be a nontrivial $A$-harmonic in $B_R$. Then there exist $\theta \in (0,1/2)$ and $C>0$ such that for all $r\in (0, R)$,
\begin{equation}\label{loggrow-interior}
    \log \frac{\| u \|_{L^\infty(B_{r})}}{ \| u \|_{L^\infty(B_{r/2})} } \le C \Big( 1+ \log \frac{R}{r} \Big)  \Big( 1 + \log \frac{\| u \|_{L^\infty(B_{R})}} { \| u \|_{L^\infty(B_{\theta R})} } \Big) .
\end{equation}
\end{corollary}
\begin{proof}
Let $r_1 \le cR$ for some small $c>0$. Let $m$ be a positive integer to be selected later. Applying \eqref{est-quasiball} $m$ times to $r = {2^{-j} r_1}$ with $j=0,1,\cdots, m-1$, we obtain
\begin{equation}\label{est-db-2mr1}
    \frac{\norm{u}_{L^\infty(\mathcal{B}_{r_1})}}{\norm{u}_{L^\infty(\mathcal{B}_{2^{-m} r_1 })}}\le \bigg( C\frac{\norm{u}_{L^\infty(\mathcal{B}_{R})}}{\norm{u}_{L^\infty(\mathcal{B}_{R/4})}} \bigg)^m.
\end{equation}
By \eqref{quasiball}, we see that
\begin{equation}\label{balls-relation}
    B_{R(\frac{r_1}{C_1R})^{1/\sigma}} \subset \mathcal{B}_{r_1} \quad \text{and} \quad \mathcal{B}_{2^{-m}r_1} \subset B_{R(\frac{C_1 2^{-m}r_1}{R})^\sigma}.
\end{equation}

Set $r = R(\frac{r_1}{C_1R})^{1/\sigma}$. We choose $m$ to be the smallest integer such that
\begin{equation}\label{findm}
    R\Big(\frac{C_1 2^{-m}r_1}{R} \Big)^\sigma \le \frac12 r = \frac12 R\Big(\frac{r_1}{C_1R} \Big)^{1/\sigma}.
\end{equation}
A simple calculation yields that
\begin{equation*}
    m = C + (\log 2)^{-1}(\sigma^{-2} - 1) \log \frac{R}{r_1} \le C\Big(1 + \log \frac{R}{r} \Big).
\end{equation*}
Moreover, by \eqref{est-db-2mr1}, \eqref{balls-relation} and \eqref{findm}, we have
\begin{equation*}
    \frac{\| u \|_{L^\infty(B_{r})}}{ \| u \|_{L^\infty(B_{r/2})} } \le \bigg( C\frac{\norm{u}_{L^\infty(B_{R})}}{\norm{u}_{L^\infty(\mathcal{B}_{R/4})}} \bigg)^m.
\end{equation*}
By noting that $B_{\theta R} \subset \mathcal{B}_{R/4}$ for some $\theta \in (0,1/2)$ and taking logarithmic to the last inequality, we obtain 
\eqref{loggrow-interior} for $r < c R$. Finally, for $cR<r<R$, the desired estimate is trivial by choosing $\theta \le c/2$.
\end{proof}

\subsection{Rigid form of doubling index}
\label{subsec:DefDoubling}

If $0 \in \overline{\Omega}$ and $A(0) = I$, then the doubling index centered at 0 of an $A$-harmonic function $u$ is defined by
\begin{equation}\label{def.DoublingIndex.A=I}
    N_u(0,r) = \log \frac{\| u \|_{L^2(B_{2r} \cap \Omega)}^2 }{\| u \|_{L^2(B_{r} \cap \Omega)}^2}.
\end{equation}
The doubling index is a crucial notation in the study of unique continuation due to its monotonicity formula (for the interior case), provided that $A$ is Lipschitz. Hence it is useful and convenient to introduce a rigid form of this quantity for general coefficient matrix $A(x)$ which is invariant under an affine transformation; see e.g. \cite{ZZ23} for details.
For any $x_0\in \overline{\Omega}$, we define
\begin{equation*}
\begin{aligned}
\mu (x_0,y) &= \frac{(y-x_0)\cdot A(x_0)^{-1}A(y)A(x_0)^{-1}(y-x_0)}{(y-x_0)\cdot A(x_0)^{-1}(y-x_0)}, \\
E(x_0,r) &= x_0+A^{\frac{1}{2}}(x_0)(B_r(0)),
\end{aligned}
\end{equation*}
and 
\begin{equation}\label{defJ2}
    J_{u}(x_0,r) = |\mathrm{det} A(x_0)|^{-\frac 12}\int_{ E(x_0,r)\cap \Omega}\mu(x_0,y)u^2dy.
\end{equation}

The doubling index is defined by
\begin{equation*}
    N_{u}(x_0,r)=  \log \frac{J_{u}(x_0,2r)}{J_{u}(x_0,r)}.\\
\end{equation*}
Note that if $A(x_0) = I$, then the above definition is reduced to \eqref{def.DoublingIndex.A=I}.

Recall the well-known almost monotonicity of the interior doubling index, i.e., if $u$ is $A$-harmonic in $E(x_0,2r)$, then
\begin{equation*}
    \sup_{0<s<r} N_u(x_0,s) \le CN_u(x_0,r).
\end{equation*}

Notice that the underlying domain $E_r(x_0) = E(x_0, r)$ of integration in \eqref{defJ2} is an ellipsoid that may vary from point to point for variable coefficients. Nevertheless, they are comparable to the Euclidean balls $B_r(x_0) = B(x_0, r)$ in the sense that
\begin{equation*}
    B_{\Lambda^{-1/2} r}(x_0) \subset E_r(x_0) \subset B_{\Lambda^{1/2} r}(x_0).
\end{equation*}
Without loss of generality, in a small region we may even assume $\Lambda \in [1, 2)$ by an affine transformation and the Lipschitz continuity of $A(x)$.

\subsection{From boundary doubling index to interior doubling index}

For the Neumann boundary value problem, it seems difficult to directly establish the doubling property for the interior points close to the boundary. Our strategy is to pass the estimate of doubling index from the boundary to the interior, by reducing the Neumann boundary value problem to the interior problem, which is possible since Corollary \ref{coro-interior-log} does not rely on the regularity of the coefficients.

Let $\Omega$ be a bounded Lipschitz domain.
Let $0\in \partial \Omega$ and $u$ be $A$-harmonic in $B_R \cap \Omega$, subject to the vanishing Neumann boundary condition. For $0< R< r_0$, there exists a bi-Lipschitz diffeomorphism $F^{-1} : B_R \cap \Omega \to \R^2_+:=\{(x_1,x_2): x_2>0 \}$ such that $F^{-1}(0) = 0$ and $F^{-1}(B_R \cap \partial \Omega) \subset \partial \R^2_+$. Moreover, $F$ locally distorts distances by a bounded factor, i.e., $$C^{-1}|x-y|\le |F^{-1}(x-y)|\le C |x-y|$$ 
for all $x, y \in B_R \cap \Omega$. This particularly implies that
\begin{equation*}
    B_{C^{-1} r} \cap \R^2_+ \subset F^{-1}(B_r\cap\Omega ) \subset B_{Cr}\cap \R^2_+ .
\end{equation*}
Define $\tilde{u} = u\circ F$. Then $\tilde{u}$ is $\widetilde{A}$-harmonic in $B_{C^{-1}R} \cap \R^2_+ \subset F^{-1}(B_R) \cap {\R^2_+}$ with
\begin{equation*}
    \widetilde{A} = |\det \mathcal{J}_F| \mathcal{J}_F^{-T} (A\circ F) \mathcal{J}_F^{-1},
\end{equation*}
where $\mathcal{J}_F$ is the Jacobian of $F$. Moreover, $\tilde{u}$ satisfies the vanishing Nuemman boundary condition on the flat boundary $F^{-1}(B_R \cap \partial \Omega)$.

Next, we extend the function $\tilde{u}$ to the entire ball by an even reflection across the flat boundary $\{ x_2=0 \}$, i.e., $\widetilde{u}(x_1, -x_2) = \widetilde{u}(x_1, x_2)$ for $x_2<0$. Meanwhile, we extend the coefficient matrix $\widetilde{A}(x_1,x_2)$ for $x_2 < 0$ by
\begin{equation*}
    \widetilde{A}(x_1, x_2) =  \begin{pmatrix}
 1 &0 \\0
  &-1
\end{pmatrix} \widetilde{A}(x_1, -x_2) \begin{pmatrix}
 1 &0 \\0
  &-1
\end{pmatrix}.
\end{equation*}

With the extended coefficient matrix, still denoted by $\widetilde{A}$, the function $\tilde{u}$ (now well-defined on the entire ball $B_{C^{-1} R}$) is $\widetilde{A}$-harmonic and $\widetilde{A}$ is bounded measurable and satisfies the ellipticity condition. Observe that in general, $\widetilde{A}$ may have a jump (thus discontinuous) across the flat interface $\{ x_2 = 0 \}$. Nevertheless, Corollary \ref{coro-interior-log} applies to $\tilde{u}$, which leads to the following result.

\begin{proposition}\label{prop-Nu-loggrowth}
    Let $\Omega\subset \mathbb{R}^2$ be a Lipschitz domain and $0\in \partial\Omega.$  Suppose that $u$ is $A$-harmonic in $B_R(0) \cap \Omega$ subject to the vanishing Neumann boundary condition on $B_R(0)\cap\partial\Omega.$ Then  there exist $\theta \in (0,1/2), C>0$ and $r_0>0$ depending on $A$ and $\Omega$ such that if $0<R<r_0  $, then for $r\in (0,\Lambda^{-1/2} R)$,
    \begin{equation}\label{est-Nu-log}
        N_u(0, r) \le C \Big( 1+ \log \frac{R}{r} \Big) \log \frac{\| u\|_{L^2(B_R \cap \Omega )}}{ \| u \|_{L^2(B_{\theta R}\cap \Omega)}} .
    \end{equation}
\end{proposition}
\begin{proof}
    As discussed before, by setting $\tilde{u} = u\circ F$ and reflection, $\tilde{u}$ is $\widetilde{A}$-harmonic in $B_{C^{-1}R}$. By Corollary \ref{coro-interior-log}, we have, for $0<r<\frac12 C^{-1}R$,
    \begin{equation}\label{loggrow-interior2}
    \log \frac{\| \tilde{u} \|_{L^\infty(B_{r})}}{ \| \tilde{u} \|_{L^\infty(B_{r/2})} } \le C \Big( 1+ \log \frac{R}{r} \Big) \Big( 1 +  \log \frac{\| \tilde{u} \|_{L^\infty(B_{ C^{-1} R/2})}} { \| \tilde{u} \|_{L^\infty(B_{\theta C^{-1} R/2})} } \Big).
    \end{equation}
In view of the bi-Lipschitz diffeomorphism $F$, we have
\begin{equation*}
    \| \tilde{u} \|_{L^{\infty}(B_{C^{-1}r})} \le  \| u \|_{L^{\infty}(B_r \cap \Omega) } \le \| \tilde{u} \|_{L^{\infty}(B_{Cr}) }.
\end{equation*}
Let $C_1$ be a constant to be selected.
Hence by using \eqref{loggrow-interior2} multiple times, we have
\begin{equation*}
\begin{aligned}   
    \log \frac{\| u \|_{L^\infty(B_{C_1 r} \cap \Omega)}}{ \| u \|_{L^\infty(B_{r} \cap \Omega)} } & \le \log \frac{\| \tilde{u} \|_{L^\infty(B_{C_1 Cr})}}{ \| \tilde{u} \|_{L^\infty(B_{C^{-1}r})} } \le \sum_{j=0}^K \log \frac{\| \tilde{u} \|_{L^\infty(B_{2^{j+1} C^{-1} r})}}{ \| \tilde{u} \|_{L^\infty(B_{2^{j}C^{-1}r})} } \\
    & \le C_2 \Big( 1+ \log \frac{R}{r} \Big) \Big( 1 +  \log \frac{\| \tilde{u} \|_{L^\infty(B_{C^{-1} R/2})}} { \| \tilde{u} \|_{L^\infty(B_{\theta C^{-1} R/2})} } \Big) \\
    & \le C_2 \Big( 1+ \log \frac{R}{r} \Big) \Big( 1 +  \log \frac{\| u \|_{L^\infty(B_{R/2} \cap \Omega)}} { \| u \|_{L^\infty(B_{\theta C^{-2} R/2} \cap \Omega)} } \Big),
\end{aligned}
\end{equation*}
where $K$ is the smallest integer such that $2^{K+1}C^{-1} \ge CC_1$.

Now, by the elliptic regularity, 
\begin{equation*}
    J_u(0,C_1 \Lambda^{-1} r) \le C\| u \|_{L^\infty(B_{C_1 r} \cap \Omega)},
\end{equation*}
and
\begin{equation*}
    J_u(0,2\Lambda r) \ge C^{-1}\| u \|_{L^\infty(B_{r} \cap \Omega)},
\end{equation*}
Choosing $C_1 = 4\Lambda^2$ and combining the previous estimates, we have
\begin{equation*}
\begin{aligned}
    N_u(0,2\Lambda r) & = \log \frac{J_u(0,4\Lambda r)}{J_u(0,2\Lambda r)} \le C+ \log \frac{\| u \|_{L^\infty(B_{C_1 r} \cap \Omega)}}{ \| u \|_{L^\infty(B_{r} \cap \Omega)} } \\
    & \le C \Big( 1+ \log \frac{R}{r} \Big) \Big( 1 +  \log \frac{\| u \|_{L^2(B_{R} \cap \Omega)}} { \| u \|_{L^2(B_{\theta C^{-3} R} \cap \Omega)} } \Big) \\
    & \le C \Big( 1+ \log \frac{R}{r} \Big) \log \frac{\| u \|_{L^2(B_{R} \cap \Omega)}} { \| u \|_{L^2(B_{\theta C^{-3} R} \cap \Omega)} },
\end{aligned}
\end{equation*}
where in the last inequality the constant 1 is absorbed due to Proposition \ref{prop.lowbd.doubling}.
After relabeling $\theta C^{-3}$ as $\theta$ and $2\Lambda r$ as $r$, we obtain the desired \eqref{est-Nu-log} for $0<r<cR$ for some $c>0$. Finally, the estimate for the range $cR<r<\Lambda^{-1}R$ is trivial by requiring $\theta \le c\Lambda^{-1}$.
\end{proof}

The next proposition allows us to pass the boundary doubling index to the interior doubling index. This will be useful when we estimate the length of nodal curves quantitatively.
\begin{proposition}\label{prop.IntDoubling}
    Let $\Omega, A$ and $u$ satisfy the same assumptions as Proposition \ref{prop-Nu-loggrowth}. In addition, assume $A$ satisfies \eqref{Lip}. Then for every $x\in B_{\theta R} \cap \Omega$ with $d(x):= {\rm dist}(x,\partial \Omega) \ge \frac12 |x|$ (i.e., $x$ is in a nontangential cone of the origin) and $r\in (0, c R)$,
    \begin{equation}\label{est-NuInt-log}
        N_u(x, r) \le C \Big( 1+ \log \frac{R}{\max{\{d(x), r\} } } \Big) \log \frac{\| u\|_{L^2(B_R \cap \Omega )}}{ \| u \|_{L^2(B_{\theta R} \cap \Omega)}} .
    \end{equation}
\end{proposition}

\begin{proof}
    First note that $0 \in \partial \Omega$ and thus $\frac12|x| \le d(x) \le |x|$. Now we consider three cases.
    
    \textbf{Case 1:} $r \in (4 d(x), cR)$. Then
    \begin{equation*}
        E(0, \Lambda^{-1} r/4) \subset E(x, r) \subset E(x, 2r) \subset E(0, 4 \Lambda r).
    \end{equation*}
    Thus, by Proposition \ref{prop-Nu-loggrowth},
    \begin{equation}\label{est.IntNu.logr}
    \begin{aligned}
        N_u(x,r) & \le \log \frac{J_u(0, 4\Lambda r)}{J_u(0, \Lambda^{-1} r/4)} \le \sum_{j=0}^{3+ [2\log_2 \Lambda]} N_u(0, 2^{j-2} \Lambda^{-1} r) \\
        & \le C \Big( 1+ \log \frac{R}{r} \Big)  \log \frac{\| u\|_{L^2(B_R \cap \Omega)}}{ \| u \|_{L^2(B_{\theta R} \cap \Omega )}}.
    \end{aligned}
    \end{equation}

    \textbf{Case 2:} $r\in (\Lambda^{-1} d(x)/4, 4 d(x))$. For convenience, set $R_x = 4d(x)$. In this case, we apply the argument in Proposition \ref{prop-Nu-loggrowth}, using Corollary \ref{coro-interior-log}. In fact, by flatting the boundary and even reflection, as well as Lemma \ref{lem.quasiball}, we can show, for any $r\in (\Lambda^{-1} R_x/16,  R_x)$,
    \begin{equation*}
    \begin{aligned}
        N_u(x,r) & \le C\Big (1 + \log \frac{R_x}{r} \Big) \log \frac{\| u \|_{L^2(B_{\theta^{-1} \Lambda R_x}(x) \cap \Omega) } }{ \| u \|_{L^2(B_{\Lambda R_x} (x) \cap \Omega ) }  }  \\
        & \le C \log \frac{\| u \|_{L^2(B_{\theta^{-1} \Lambda R_x}(x) \cap \Omega) } }{ \| u \|_{L^2(B_{\Lambda R_x} (x) \cap \Omega) }  } \\
        & \le C \Big( 1+ \log \frac{R}{d(x)} \Big)  \log \frac{\| u\|_{L^2(B_R \cap \Omega )}}{ \| u \|_{L^2(B_{\theta R} \cap \Omega)}},
    \end{aligned}
    \end{equation*}
    where we have used the result from Case 1 for a finite number of times for $4d(x) \le r \le 4\theta^{-1}\Lambda^2 d(x)$.

    \textbf{Case 3:} $r \in (0, \Lambda^{-1} d(x)/4)$.  In this case, we have $E(x,4r) \subset \Omega$. As a consequence of the interior almost monotonicity of the doubling index for $r\in (0, \Lambda^{-1} d(x)/4)$ (see \cite{HLbook}), we have
    \begin{equation*}
        N_u(x ,r) \le CN_u(x, \Lambda^{-1} d(x)/2) \le C \Big( 1+ \log \frac{R}{d(x)} \Big)  \log \frac{\| u\|_{L^2(B_R \cap \Omega )}}{ \| u \|_{L^2(B_{\theta R} \cap \Omega)}}.
    \end{equation*}
    where in the last inequality we used the result from the second case when $r =\Lambda^{-1} d(x)/2$. This ends the proof.
\end{proof}

\begin{remark}
    We will not need the full version of the above proposition in this paper. Indeed, we only need the case when $R \simeq d(x) \simeq r$.
\end{remark}

\section{Critical points}
\label{sec:criticalpoint}

It has been well-known that in the two dimensional case, the critical set of an $A$-harmonic function consists of discrete points in the interior domains. We refer to the recent work \cite{Zhu23b,KZ25} containing the estimate of the number of interior critical points for the two dimensional equations. On the other hand, the number of boundary critical points depends essentially on the boundary geometry. In particular, for a Lipschitz boundary, each convex tip will generate a critical point and therefore the number of critical points can be infinity even if the boundary is convex; also see the counterexample in \cite{KZ22}. Nevertheless, if the boundary is flat enough, say $C^{1,\alpha}$ or even $C^{1,{\rm Dini}}$, then we can actually show the number of boundary critical points is finite.  

In this paper, we will use an geometric approach to estimate the number the critical points in planar domains, which is inspired by \cite{Ale87,AM94}. The following lemma is a restatement of \cite[Remark 1.2]{Ale87} for Lipschitz coefficients and \cite[Lemma 2.5]{AM94} (including its following remark), without referring to the notation of geometric index. We provide an independent and self-contained proof for the reader's convenience.

\begin{lemma}\label{lem.Geo.CP}
    Let $A$ be $C^\alpha$-H\"{o}lder continuous. Let $u$ be $A$-harmonic in $\Omega \subset \R^2$. Then $\nabla u(x) = 0$ for $x\in \Omega$ if and only if $x$ is an intersection of at least two level curves.
\end{lemma}
\begin{proof}
    The ``if'' part follows from the implicit function theorem.
    We only need to prove the ``only if'' part. Without loss of generality, assume $u(0) = 0, \nabla u(0) = 0$ and $A(0) = I$. It suffices to show that $0$ is an intersection of at least two level curves.

 \textbf{Step 1}: Prove it for harmonic functions. Actually, if $u$ is harmonic and $u(0) = 0$, $\nabla u(0) = 0$, then in polar coordinates,
    \begin{equation*}
        u(x) = a_m r^m \cos(m\theta + \varphi_m) + O(r^{m+1}),
    \end{equation*}
    for some $m \ge 2$ and $a_m \neq 0$. Definitely, this implies that in a neighborhood of $0$, there are exactly $m$ nodal curves intersecting at $0$.

    \textbf{Step 2}: A blow up argument for $A$-harmonic functions.  It is well-known that in two dimensional case, we have the strong unique continuation at the interior points. This means that there exists a sequence of $r_k \to 0$ such that $N_u(0,r_k) \le C$ for some constant $C$. For otherwise, the vanishing order will be infinity. 

 Then define
    \begin{equation*}
        u_k(x) = \frac{u(r_k x)}{\| u(r_k\cdot) \|_{L^2(B_2)}}.
    \end{equation*}
    Since $u_k$ is uniformly $C^{1,\alpha}$ in $B_{3/2}$,
    we can find a subsequence of $\{ u_k\}$, still denoted by $\{ u_k \}$, such that $u_k \to u_\infty$ strongly in $C^1(B_1)$ and $u_\infty$ is nontrivial. Moreover, it is straightforward to verify that $u_\infty (0) = 0, \nabla u_\infty(0) = 0$ and $\Delta u_\infty = 0$ in $B_1$. By Step 1, we see that $u_\infty$ has at least two level curves intersecting at $0$. This also implies that $u_\infty$ at least has 4 disjoint nodal components, i.e., the subdomain of $B_1$ bounded by the nodal curves and $\partial B_1$. The function $u_\infty$ does not change sign on each nodal component. By the uniform convergence from $u_k \to u_\infty$, we see that for $k$ sufficiently large, $u_k$ also has at least 4 nodal components, which is equivalent to say that $u(x)$ has at least 4 nodal components in $B_{r_k}$ for $r_k$ small enough. This implies that $0$ must be an intersection of at least two nodal curves of $u$.
\end{proof}

\begin{proof}[Proof of Theorem \ref{thm.criticalpts}]
    Let $v$ be the stream function of $u$ in $\Omega \cap B_R$. Observe that \eqref{eq.Stream} implies that $|\nabla u| \simeq |\nabla v|$, and the critical points of $u$ are exactly the critical points of $v$. Thus it suffices to estimate the number of critical points of $v$ in $\overline{\Omega} \cap B_{\theta R}$. The reason that we estimate the number of critical points of $v$ instead of $u$ is explained in Remark \ref{rmk.u2v}.
    
    Recall that $v$ satisfies $\nabla\cdot (\widehat{A} \nabla v) = 0$ in $\Omega \cap B_R$ with $\widehat{A} = (\det A)^{-1} A$ and the vanishing Dirichlet boundary condition $v = 0$ on $\partial \Omega \cap B_R$. If $A$ is $C^\alpha$-H\"{o}lder continuous (resp. Lipschitz), then $\widehat{A}$ is also $C^\alpha$-H\"{o}lder continuous (resp. Lipschitz) with possibly different constants.

 Let $B_{R,+} = B_R \cap \{ x_2 > 0\}$. Since $\partial \Omega$ is $C^{1,\alpha}$, there exists a $C^{1,\alpha}$ diffeomorphism $F$ such that $F: B_{R,+}(0) \to \Omega \cap B_R$ such that $F(B_{R,+}) \subset \Omega \cap B_{R/2}$ and $F( B_R \cap \{ x_2 = 0 \}) \subset \partial \Omega \cap B_R$. Let $\tilde{v}(x) = v\circ F(x)$ for $x\in B_{R,+}$. Then $\tilde{v}$ is $\widetilde{A}$-harmonic in $B_{R,+}$, where
    \begin{equation*}
        \widetilde{A} = |\det \mathcal{J}_F| \mathcal{J}_F^{-T} (\widehat{A} \circ F) \mathcal{J}_F^{-1}.
    \end{equation*}
    Moreover, $\widetilde{A}$ is $C^\alpha$ continuous and $\tilde{v}  = 0$ on the flat boundary $B_R \cap \{ x_2 = 0\}$. Due to the bi-$C^{1,\alpha}$ diffeomorphism, $\nabla v\circ F(x) = 0$ if and only if $\nabla \tilde{v}(x) = 0$.

For the problem with vanishing Dirichlet boundary condition, we perform an odd reflection and thus define $\tilde{v}$ as an $\widetilde{A}$-harmonic function in $B_R$, where $\widetilde{A}$ is extended to $B_R$ by
    \begin{equation*}
    \widetilde{A}(x_1, x_2) =  \begin{pmatrix}
 1 &0 \\0
  &-1
\end{pmatrix} \widetilde{A}(x_1, -x_2) \begin{pmatrix}
 1 &0 \\0
  &-1
\end{pmatrix} \qquad \text{for } x \in B_{R,-} = B_R \cap \{ x_2 < 0\}.
\end{equation*}
The key point is that $\tilde{v}$ is still $C^{1,\alpha}$ in $B_{R}$, though $\widetilde{A}$ is not continuous on the interface $\{ x_2 = 0\}$. This is because $\tilde{v}$ is $C^{1,\alpha}$ in $B_{R,+}$ up to the boundary $B_R \cap \{ x_2 = 0\}$ with vanishing Dirichlet boundary condition, and an odd reflection does not generate a jump for $\nabla \tilde{v}$ on the interface $\{ x_2 = 0 \}$.

    Using quasiconformal mapping, we write $\tilde{v} = w\circ \chi$, where $w$ is harmonic in $B_{R/2}$ and $\chi: B_{R/2} \to B_{R/2}$ satisfying $\chi(0) = 0$ and $\chi(\partial B_{R/2}) = \partial B_{R/2}$. By the classical estimate of critical points for harmonic functions (see e.g. \cite{Han07}), we have
    \begin{equation*}
        \# \{ \mathcal{C}(w) \cap B_{cR} \} \le C \log \frac{\| w \|_{L^2(B_{R/2})}}{ \| w \|_{L^2(B_{R/4})}}.
    \end{equation*}
    The critical points of $w$ are known as geometric critical points of $\tilde{v}$; see the definition in \cite{AM94}. By Lemma \ref{lem.Geo.CP}, each interior critical point of $\tilde{v}$ in $B_{\theta R} \setminus \{ x_2 = 0\}$ will generate a critical point of $w$ in $B_{cR}$, where $\theta$ is chosen such that $\chi(B_{\theta R}) \subset B_{cR}$ . In fact, if $x_0 \in B_{\theta R,+}$ is a critical point of $\tilde{v}$, Lemma \ref{lem.Geo.CP} implies that $x_0$ is an intersection of at least two level curves. Because the quasiconformal mapping $\chi$ as a homeomorphism preserves the level curves, $\chi(x_0)$ is an intersection of at least two level curves of $w$. Then by Lemma \ref{lem.Geo.CP} again, $\chi(x_0) \in B_{cR}$ is a critical point of $w$.

    Next, we consider the critical points of $\tilde{v}$ on the interface $B_{\theta R} \cap \{x_2 = 0 \}$. Since $A$ is discontinuous on the interface, we cannot apply Lemma \ref{lem.Geo.CP} directly. Instead, we will use the fact that $\tilde{v} = 0$ on $\{ x_2 = 0 \}$, which is a level curve. Now if $\nabla \tilde{v}(x_0) = 0$ for $x_0\in B_{\theta R} \cap \{ x_2=0 \}$, we claim that there exists another level curve touch $x_0$ from $B_{\theta R, +}$. For otherwise, $\tilde{v}$ does not change sign in a neighborhood of $x_0$ above $\{ x_2 = 0\}$. By the Hopf's lemma, this will implies that $\partial_2 \tilde{v}(x_0) \neq 0$, leading to a contradiction. Consequently, $\tilde{v}$ has two level curves intersecting at the critical point $x_0 $ on the interface. Then as before, by Lemma \ref{lem.Geo.CP}, $\chi(x_0)$ is an intersection of two level curves of $w$ which implies that $\chi(x_0)$ is a critical point of $w$ in $B_{c R}$.

  It follows that
    \begin{equation*}
       \# \{ \mathcal{C}(\tilde{v}) \cap B_{\theta R} \} \le \# \{ \mathcal{C}(w) \cap B_{cR} \}.
    \end{equation*}
    Consequently, for some $\theta \in (0,1)$,
    \begin{equation}\label{est.critByw}
        \# \{ \mathcal{C}(u) \cap B_{\theta R} \} = \# \{ \mathcal{C}(v) \cap B_{\theta R} \} \le \# \{ \mathcal{C}(\tilde{v}) \cap B_{\theta _1 R} \} \le C \log \frac{\| w \|_{L^2(B_{R/2})}}{ \| w \|_{L^2(B_{R/4})}}.
    \end{equation}

    Finally, we only need to bound the right-hand side of \eqref{est.critByw} by a doubling quantity in terms of $u$. Actually,
    \begin{equation*}
    \begin{aligned}
        |B_{R/2}|^{-1/2} \| w\|_{L^2(B_{R/2})} & \le \| w\|_{L^\infty(B_{R/2})} = \| \tilde{v} \|_{L^\infty(B_{R/2})} \\
        & \le CR^{-1}\| \tilde{v} \|_{L^2(B_{R,+})} \le C \| \nabla \tilde{v} \|_{L^2(B_{R,+})} \\
        &  \le C \| \nabla v \|_{L^2(F(B_{R,+}))} 
     \le C \| \nabla u \|_{L^2(F(B_{R,+}))} \\    & \le C \| \nabla u \|_{L^2(\Omega \cap B_{R})}, 
 \end{aligned}    
    \end{equation*}
    where we used the fact $|\nabla v| \simeq |\nabla u|$ in $\Omega \cap B_R$. 
Similarly, we can derive
\begin{equation*}
   \begin{aligned}
    \| \nabla u \|_{L^2(\Omega \cap B_{\theta R})} &\le C  \| \nabla v\|_{L^2(\Omega \cap B_{\theta R})}
    \leq C R^{-1}\| \tilde{ v}\|_{L^2(\Omega \cap B_{2\theta R})}\\ 
   & \leq C|B_{R/4}|^{-1/2} \| w\|_{L^2(B_{R/4})},
    \end{aligned} 
\end{equation*}
for some $\theta \in (0,1)$ small enough (depending only on $A$ and $\Omega$). Thus,
    \begin{equation*}
        \log \frac{\| w \|_{L^2(B_{R/2})}}{ \| w \|_{L^2(B_{R/4})}} \le C \log \frac{\| \nabla u \|_{L^2(\Omega \cap B_{R})}}{ \| \nabla u \|_{L^2(\Omega \cap B_{\theta R})}}.
    \end{equation*}
    This, together with \eqref{est.critByw}, finishes the proof.
\end{proof}

\begin{remark}\label{rmk.u2v}
    In this remark, we explain why we study the critical points of the stream function $v$, instead of the original $A$-harmonic function $u$.
    The argument in the above proof does not work well for the boundary critical points of $u$ with a vanishing Neumann boundary condition. The reason is that, for the Neumann boundary condition, we perform an even reflection and thus the gradient is discontinuous. Moreover, the boundary itself is not a level curve and there might be only one level curve connecting to a critical point on the boundary. Therefore such critical points are not geometric critical points.
\end{remark}

\begin{remark}\label{rmk.intCP}
    If the boundary is merely Lipschitz, by argument in the above proof, we can still get the optimal bound of the number of interior critical points, namely, 
    \begin{equation*}
        \# (\{ x\in \Omega \cap B_{\theta R}: \nabla u(x) = 0 \}) \le C\log \frac{\| \nabla u \|_{L^2(\Omega \cap B_{R})}}{ \| \nabla u \|_{L^2(\Omega \cap B_{\theta R})}}.
    \end{equation*}
    However, the argument for the critical points on $\partial \Omega \cap B_{\theta R}$ does not work any more, simply because of the failure of Hopf's lemma on general Lipschitz boundaries. As we have mentioned earlier, in this case the critical points may accumulate on the boundary to an infinite number.
\end{remark}

\section{Estimate of doubling index}
\label{sec:DoublingIndex}

\subsection{Doubling index of stream functions}

 In this section, assume either $\Omega$ is $C^{1,\alpha}$ or convex. Let $A$  satisfy \eqref{ellipticity}, \eqref{symmetry} and \eqref{Lip}.
 Suppose $u$ is a nontrivial $A$-harmonic function in $B_R(0)\cap\Omega$,  satisfying the vanishing Neumann boundary condition on \(B_R(0) \cap \partial \Omega\). The doubling index of the $A$-harmonic functions has been defined in Subsection \ref{subsec:DefDoubling}.

Recall that the stream function $v$ satisfies the elliptic equation $-\nabla\cdot (\widehat{A} \nabla v) = 0$ in $\Omega \cap B_R$, subject to the vanishing Dirichlet boundary condition, where $\widehat{A} = (\det A)^{-1} A$. Then we can define the doubling index for $v$, similar to $u$, but with $A$ replaced by $\widehat{A}$. Precisely, we define the notation $\widehat{\mu}, \widehat{E}, \widehat{J}_{}, \widehat{N}_{}$ as follows:
\begin{equation*}
\begin{aligned}
     \widehat{\mu}(x_0,y)&=  \frac{(y-x_0)\cdot \widehat{A}(x_0)^{-1}\widehat{A}(y)\widehat{A}(x_0)^{-1}(y-x_0)}{(y-x_0)\cdot \widehat{A}(x_0)^{-1}(y-x_0)}, \\
     \widehat{E}(x_0,r)&= x_0+\widehat{A}^{\frac{1}{2}}(x_0)(B_r(0)),\\
    \widehat{J}_{v}(x_0,r)&=|\mathrm{det} \widehat{A}(x_0)|^{-\frac 12}\int_{\widehat{E} (x_0,r)\cap \Omega}\widehat{\mu}(x_0,y)v^2dy,\\
    \widehat{N}_{v}(x_0,r) & =\log \frac{\widehat{J}_{v}(x_0,2r)}{\widehat{J}_{v}(x_0,r)}.
\end{aligned}
\end{equation*}

Since $\widehat{A}$ is just a scalar multiple of $A$, we have the following relations.
\begin{equation*}
\begin{aligned}
\mu (x_0,y)=& \widehat{\mu}(x_0,y),\\
E(x_0,r)=&\widehat{E}(x_0,(\det A(x_0))^{-1/2}r),\\
    J_{u}(x_0,r)=&
   \det A(x_0) \widehat{J}_{u}(x_0,(\det A(x_0))^{-1/2}r),\\
   \widehat{N}_{u}(x_0,r) =& N_{u}(x_0, (\det A(x_0))^{-1/2} r).
\end{aligned}
\end{equation*}
The above relations show that the coefficient matrices $A$ and $\widehat{A}$ share the same form of the doubling index, up to a harmless dilation.

Finally we define the gradient doubling index by
\begin{equation*}
    N_{\nabla u}(x, r) = \log \frac{\| \nabla u \|_{L^2(E(x,2r) \cap \Omega)}}{\| \nabla u \|_{L^2(E(x,r)\cap \Omega)}}.
\end{equation*}
The definition for $N_{\nabla v}(x,r)$ is defined in the same way.

\subsection{Almost monotonicity}

In this subsection, we will establish the almost monotonicity of the boundary doubling index, which is useful in the estimate of level curves. We begin with two lemmas comparing the doubling index with gradient doubling index.
\begin{lemma}\label{lem.doubling.u<Du}
    For any $x\in \overline{\Omega} \cap B_R$ with $E_{2r}(x) \subset B_R$,
    \begin{equation}\label{est.u<Du}
        \frac{\| u \|_{L^2(E_{2r}(x) \cap \Omega)}} {\| u \|_{L^2(E_{r}(x) \cap \Omega)} } \le C\frac{\| \nabla u \|_{L^2(E_{2r}(x) \cap \Omega)}} {\| \nabla u \|_{L^2(E_{r/2}(x) \cap \Omega)} }.
    \end{equation}
 \end{lemma}

\begin{proof}
Fix $x\in \overline{\Omega} \cap B_R$ such that $E_{2r}(x) \subset B_R$.
Using the triangle inequality and the elliptic regularity for Neumann boundary value problem, we have
\begin{equation*}
\begin{aligned}
    \| u \|_{L^2(E_{2r}(x) \cap \Omega)} & \le \| u - u(x) \|_{L^2(E_{2r}(x) \cap \Omega)} + |E_{2r}(x) \cap \Omega|^{1/2} |u(x)| \\
    & \le \| u - u(x) \|_{L^2(E_{2r}(x) \cap \Omega)} + C \| u \|_{L^2(E_r(x) \cap \Omega)}.
\end{aligned}
\end{equation*}
It follows that
\begin{equation}\label{est.dbl2Du}
    \frac{\| u \|_{L^2(E_{2r}(x) \cap \Omega)}} {\| u \|_{L^2(E_{r}(x) \cap \Omega)} } \le \frac{\| u - u(x) \|_{L^2(E_{2r}(x) \cap \Omega)}} {\| u \|_{L^2(E_{r}(x) \cap \Omega)} } + C.
\end{equation}

By the regularity of elliptic equation under Neumann boundary condition, we have
\begin{equation*}
    \| u - u(x) \|_{L^\infty(E_r(x) \cap \Omega)} \le C \| \nabla u \|_{L^2(E_{2r}(x) \cap \Omega)}.
\end{equation*}

Then, we apply the Poincar\'{e} inequality to get
\begin{equation*}
\begin{aligned}
    \| u - u(x) \|_{L^2(E_{2r}(x) \cap \Omega)} & \le \| u - \fint_{E_r(x)\cap \Omega} u \|_{L^2(E_{2r}(x) \cap \Omega)} + Cr|u(x) - \fint_{E_r(x)\cap \Omega} u| \\
    & \le Cr \| \nabla u \|_{L^2(E_{2r}(x) \cap \Omega)}.
\end{aligned}
\end{equation*}

On the other hand, the Caccioppoli inequality implies that
\begin{equation*}
    \| u \|_{L^2(E_{r}(x) \cap \Omega)} \ge c r \| \nabla u \|_{L^2(E_{r/2}(x) \cap \Omega)}.
\end{equation*}
Taking the last two inequalities into \eqref{est.dbl2Du}, we obtain 
\begin{equation*}
    \frac{\| u \|_{L^2(E_{2r}(x) \cap \Omega)}} {\| u \|_{L^2(E_{r}(x) \cap \Omega)} } \le C\frac{\| \nabla u \|_{L^2(E_{2r}(x) \cap \Omega)}} {\| \nabla u \|_{L^2(E_{r/2}(x) \cap \Omega)} } + C \le C\frac{\| \nabla u \|_{L^2(E_{2r}(x) \cap \Omega)}} {\| \nabla u \|_{L^2(E_{r/2}(x) \cap \Omega)} },
\end{equation*}
as desired, where in the last inequality, the additional constant $C$ has been absorbed due to Proposition \ref{prop.lowbd.NDu}.
\end{proof}

\begin{lemma}\label{lem.doubling.Du<u}
    Let $x\in\overline{\Omega} \cap B_R$ with $E_{2r}(x) \subset B_R$, and $u(y) = 0$ for some $y \in \overline{\Omega} \cap E_{r/2}(x) $. Then,
    \begin{equation}\label{est.Du<u}
         \frac{\| \nabla u \|_{L^2(E_{2r}(x) \cap \Omega)}} {\| \nabla u \|_{L^2(E_{r}(x) \cap \Omega)} } \le C \frac{\| u \|_{L^2(E_{4r}(x) \cap \Omega)}} {\| u \|_{L^2(E_{r}(x) \cap \Omega)} }.
    \end{equation}
 \end{lemma}
 \begin{proof}
     This is similar to the proof of Lemma \ref{lem.doubling.u<Du} by the Caccioppoli inequality and Poincar\'{e} inequality.
 \end{proof}

Now recall the almost monotonicity formula for the elliptic equation with vanishing Dirichlet boundary condition. Particular, consider the stream function $v$, which is $\widehat{A}$-harmonic in $\Omega \cap B_R$ and $v = 0$ on $\partial \Omega \cap B_R$. By \cite[Remark 3.6]{ZZ23}, for either $\Omega$ is $C^{1,\alpha}$ or convex, we have
\begin{equation}\label{est.v.almostMnty}
    \sup_{0<s<r} \widehat{N}_v(x, s) \le C \widehat{N}_v(x,r),
\end{equation}
whenever $x\in \partial \Omega$ and $E(x,2r) \subset B_R$. By examining the proof in \cite{ZZ23}, we notice that the doubling index may be replaced by any $\lambda$-index with a fixed $\lambda > 1$, i.e., replacing $\widehat{N}_v(x,r)$ in \eqref{est.v.almostMnty} by 
\begin{equation*}
    \log \frac{\widehat{J}_v(x,\lambda r)}{\widehat{J}_v(x,r)}.
\end{equation*}
Also, notice that the factor 2 in \eqref{est.u<Du} and \eqref{est.Du<u} is not essential and can be replaced by any $\lambda>1$. This trick will be used whenever needed.

The next proposition establishes the almost monotonicity of $\nabla v$, which yields the almost monotonicity of $\nabla u$.
\begin{proposition}\label{prop.DvDu.AlmMnty}
    Let $v$ be the stream function of $u$ with $v = 0$ on $\partial \Omega \cap B_R$. Then for $x\in \partial \Omega \cap B_R$ and $E_{2r}(x) \cap B_R$,
    \begin{equation*}
        \sup_{0<s<r} N_{\nabla v}(x, s) \le C N_{\nabla v}(x,r).
    \end{equation*}
    As a corollary, we have
    \begin{equation}\label{est.Du.AlmostMnty}
        \sup_{0<s<r} N_{\nabla u}(x, s) \le C N_{\nabla u}(x,r).
    \end{equation}
\end{proposition}

\begin{proof}
    Let $s\le r/2$. Then by Lemma \ref{lem.doubling.Du<u} and the almost monotonicity of $\sqrt{2}$-index,
    \begin{equation*}
        N_{\nabla v}(x,s) \le C(\widehat{N}_v(x,2s) + \widehat{N}_v(x,s)) \le C \log \frac{\widehat{J}_v(x,2r)}{\widehat{J}_v(x,\sqrt{2} r)}.
    \end{equation*}
    By Lemma \ref{lem.doubling.u<Du}, the right-hand side of the above inequality is bounded by $CN_{\nabla v}(x,r)$. Thus,
    \begin{equation*}
        \sup_{0<s\le r/2} N_{\nabla v}(x,s) \le CN_{\nabla v}(x,r).
    \end{equation*}
    Finally, the estimate for the range $s \in (r/2,r)$ follows readily by noticing $N_{\nabla v}(x,s) \le N_{\nabla v}(x,r/2) + N_{\nabla v}(x,r)$. The estimate \eqref{est.Du.AlmostMnty} follows from the fact $|\nabla u| \simeq |\nabla v|$.
\end{proof}

The almost monotonicity of $\nabla u$ then can be transferred to $u$ in the following proposition.
\begin{proposition}\label{prop.u.AlmMnty}
Let $x\in \partial \Omega \cap B_R$ and $E_{2r}(x) \subset B_R$. Then either $Z(u) \cap E_{r/2}(x) = \varnothing$, or $N_u(x,s)$ is almost monotonic for $0 < s <r$ in the sense that
\begin{equation*}
    \sup_{0 < s <r } N_u(x,s) \le C N_u(x,r),
\end{equation*}
where $C$ depends only on $A$ and $\Omega$.
\end{proposition}
\begin{proof}
    Without loss of generality, assume $x = 0$ and $A(0) = \mathbin{I}$, and thus $E_{r/2}=B_{r/2}$. 
 Thus, Lemma \ref{lem.doubling.Du<u} implies
    \begin{equation*}
        \log \frac{\| \nabla u\|_{L^2(B_{\sqrt{2}r} \cap \Omega)}}{\| \nabla u \|_{L^2(B_{r} \cap \Omega)}} \le  \log \frac{\| u\|_{L^2(B_{2r} \cap \Omega)}}{ \| u \|_{L^2(B_{r} \cap \Omega)}} + C.
    \end{equation*}
    It follows from Lemma \ref{lem.doubling.u<Du} , Proposition \ref{prop.DvDu.AlmMnty} and the above inequality that, for any $s\in (0,r/2)$,
    \begin{equation*}
    \begin{aligned}
        \log \frac{\| u\|_{L^2(B_{2s} \cap \Omega)}}{\|  u \|_{L^2(B_{\sqrt{3}s} \cap \Omega)}} 
        & \le C + \log \frac{\| \nabla u\|_{L^2(B_{2s} \cap \Omega)}}{\| \nabla u \|_{L^2(B_{\sqrt{2}s} \cap \Omega)}} \\
        & \le C + \log \frac{\| \nabla u\|_{L^2(B_{\sqrt{2}r} \cap \Omega)}}{\| \nabla u \|_{L^2(B_{r} \cap \Omega)}} \\
        & \le C + C\log \frac{\| u\|_{L^2(B_{2r} \cap \Omega)}}{ \| u \|_{L^2(B_{r} \cap \Omega)}}. \\
    \end{aligned}
    \end{equation*}
    Combining this with Proposition \ref{prop.lowbd.doubling} implies the desired estimate  .
\end{proof}

\subsection{Optimal upper bound}

In this subsection, we prove Theorem \ref{thm.Doubling}.
\begin{proposition}\label{prop.IntDoubling.Du}
    Assume $A(0) = I$. Let $u$ be $A$-harmonic in $\Omega \cap B_R$ satisfying the vanishing Neumann boundary condition on $\partial \Omega \cap B_R$. Then for $x\in \partial \Omega \cap B_{cR}$ and $r\in (0,R/8)$
    \begin{equation*}
        N_{\nabla u}(x, r) \le CN_{\nabla u}(0,R/2).
    \end{equation*}
\end{proposition}
\begin{proof}
    By choosing $c$ small such that if $x\in B_{cR} \cap \Omega$, we have $E_{R/16}(0) \subset  E_{R/8}(x)$ and $E_{R/4}(x) \subset E_{R}(0) $. Then for $r< R/8$, by using \eqref{est.Du.AlmostMnty} twice,
    \begin{equation*}
        N_{\nabla u}(x,r) \le C N_{\nabla u}(x,R/8) \le C \log \frac{\| \nabla u \|_{L^2(E_R(0) \cap \Omega}}{\| \nabla u \|_{L^2(E_{R/16}(0)\cap \Omega}} \le CN_{\nabla u}(0,R/2).
    \end{equation*}
    This obtains the desired estimate.
\end{proof}

The next lemma reduces the upper bound of interior doubling index to the boundary doubling index.
\begin{lemma}\label{lem.bdry-to-interior}
    Let $u$ be the same as Proposition \ref{prop.IntDoubling.Du}. Then there exists $c > \theta >0$ such that
    \begin{equation*}
        \sup_{\substack{x\in B_{\theta R} \cap \Omega \\ 0<r< \theta R} } N_u(x,r) \le C \sup_{\substack{x\in B_{cR}\cap \partial \Omega \atop 0< r< cR}} N_u(x, r).
    \end{equation*} 
\end{lemma}
\begin{proof}
    This is a corollary of Proposition \ref{prop.IntDoubling}. Let $x\in B_{\theta R} \cap \Omega$. Let $d(x)  = \dist(x,\partial \Omega)$ and $z$ be a point on $\partial \Omega$ such that $|x-z| = d(x)$. Clearly, $z\in B_{cR}\cap \partial \Omega$, provided that $2\theta < c$. By Proposition \ref{prop.IntDoubling}, for $r > cd(x)$, we have
    \begin{equation*}
        N_u(x,r) \le C\bigg(1 + \log \frac{c^{-1} r}{\max \{ d(x), r \}} \bigg) \log \frac{\| u \|_{L^2(B_{c^{-1}r}(z))}} {\| u \|_{L^2(B_{\theta c^{-1}r}(z))}}.
    \end{equation*}
    Let $r<\theta R$ and $\theta \le c^2$. Then
    \begin{equation}\label{est.dbl.TouchBdry}
        \sup_{cd(x) < r<\theta R} N_u(x,r) \le C\sup_{0<r<cR} N_u(z,r).
    \end{equation}
    Now, for $r< cd(x)$ (let $c$ be small such that $E(x,2cd(x) \subset \Omega$), we simply use the well-known interior bound of doubling index to get
    \begin{equation}\label{est.OffBdry}
        \sup_{0<r<cd(x)} N_u(x,r) \le CN_u(x, cd(x)).
    \end{equation}
    Combining \eqref{est.dbl.TouchBdry} and \eqref{est.OffBdry} and taking the supremum over all $x\in B_{\theta R} \cap \Omega$, we get the desired estimate.
\end{proof}

\begin{proof}[Proof of Theorem \ref{thm.Doubling}]
  It follows from Lemma \ref{lem.bdry-to-interior}, Lemma \ref{lem.doubling.u<Du} and Proposition \ref{prop.IntDoubling.Du} in order that 
\begin{equation}\label{est.BoundByGradDbl}
        \sup_{\substack{x\in B_{\theta R} \cap \Omega\\ 0<r< \theta R} } N_u(x,r) \le CN_{\nabla u}(0,R/2).
    \end{equation}
    
    Now if $u(y) = 0$ for some $y \in B_{R/8}\cap \overline{\Omega}$, then \eqref{est.UnifDoubling} is a consequence of \eqref{est.BoundByGradDbl} and Lemma \ref{lem.doubling.Du<u}.  On the other hand, if $Z(u) \cap B_{R/8} = \varnothing$, then $u$ does not change sign in $\overline{\Omega} \cap B_{R/8}$. By the Harnack inequality\footnote{This can be seen easily by flattening the Lipschitz boundary followed by an even extension, which reduces the boundary Harnack inequality to the classical interior Harnack inequality.}, we have
    \begin{equation*}
        \sup_{\substack{x\in B_{\theta R} \cap \Omega\\ 0<r< \theta R} } N_u(x,r) \le C,
    \end{equation*}
    for some universal constant $C$ depending only on $A$ and $\Omega$. Since by Proposition \ref{prop.lowbd.NDu} the right-hand side of \eqref{est.BoundByGradDbl} has a universal lower bound, it holds true also in this situation. This ends the proof.
\end{proof}

\section{Length of level curves} \label{sec:levelset}
Throughout this section, we assume $u$ and $\Omega$ are as specified in Section \ref{sec:DoublingIndex}. 
Our aim is to prove an upper bound for the total length of level curves in terms of the doubling index. Since the level set $Z_t(u) := \{ u = t\}$ is exactly the nodal curve of $u - t$ (i.e., $Z_0(u-t)$), which is still $A$-harmonic, it is sufficient to consider the nodal curve of $u$.
The proof follows the line of \cite{LMNN21} (also see \cite{ZZ23}), with modifications adapted to the Neumann boundary condition.

\subsection{Domain decomposition and maximal doubling index}

Define a standard rectangle by $Q_0=[-\frac{1}{2},\frac{1}{2})\times [-8,8).$ Throughout this section, all the boundary rectangles considered are translated and rescaled versions of $Q_0$, centered on the boundary $\partial \Omega$.  Let $\pi(Q)$ be the vertical projection of $Q$ on $\mathbb{R}.$ For example, $\pi(Q_0) = [-\frac{1}{2},\frac{1}{2}).$
Denote the side length of $Q_0$ and $\pi(Q_0)$ by $\ell_{Q_0}=\ell(\pi(Q_0))=1.$

Without loss of generality, we assume $\partial \Omega$ is locally given by $x_2 = \phi(x_1)$ with $\phi(0) = 0$, $|\phi'| \le 1$ and  $\Omega = \{(x_1, x_2): x_2 > \phi(x_1) \}$.
Let $Q$ be a boundary rectangle centered at $x_Q \in \partial \Omega$ and $Q'=\pi(Q)$.
We partition $Q'$ into $2^k$ small equal intervals $\left \{ q' \right \} $ with side length $\ell(q')=2^{-k}\ell_Q$. For each $q'$, there is a unique boundary rectangle $q$ such that $\pi(q)=q'$ and the center of $q$ lies on $\partial\Omega\cap Q.$ Furthermore, we cover $\Omega \cap \pi^{-1}(q')\cap Q\setminus q$ by at most $8\times 2^k$ disjoint squares with side length $2^{-k}\ell_Q$. Let $\mathbb{B}_k(Q)$ be the collection of boundary rectangles and $\mathbb{I}_k(Q)$ be the collection of interior squares that intersect with $\Omega\cap Q$.

There is a simple fact about the above decomposition :
\begin{equation*}
   \min_{q\in \mathbb{I}_k(Q)}\text{dist}(q,Q\cap\partial\Omega)\ge 2\ell_q . 
\end{equation*}

Next, we define the maximal doubling index of $u$. Let $\theta$ be as given in  Theorem \ref{thm.Doubling}
and $S$ be the smallest integer such that $2^S > \theta^{-1}$.

Let $Q$ be a boundary rectangle contained in $B_{\theta R}$. Denote by $\ell_Q$ the side length of $Q$. We define the maximal doubling index of $u$ in $Q$ by 
\begin{equation}\label{defmaximal}
   N_{u}^{\partial }(Q)=\sup_{x\in Q\cap\partial\Omega, 0<r<
\ell_Q}\sum_{i = 0}^{S-1}N_u(x,2^i r)=\sup_{x\in Q\cap\partial\Omega,0<r<
\ell_Q}\log\frac{\int_{E(x,2^S r)\cap\Omega}\mu u^2}{\int_{E(x,r)\cap\Omega}\mu u^2}. 
\end{equation}
Clearly, if $q\in \mathbb{B}_k(Q)$, then $N_u^\partial(q) \le N_u^\partial(Q)$.

Let $Q$ be an interior square. Denote by $\ell_Q$ the side length of $Q$. We define the interior maximal doubling index of $u$ in $Q$ by 
\begin{equation*}
   N_{u}^*(Q)=\sup_{x\in Q,0<r<
\ell_Q}N_u(x,r) . 
\end{equation*}

\subsection{Drop of doubling index versus absence of nodal curves}
We restrict ourselves in $B_{\theta R} \cap \Omega$. Let $Q$ be a boundary rectangle centered at a point on $B_{\theta R} \cap \partial \Omega$ and contained in $B_{\theta R}$. 
The following property of boundary maximal doubling index is in the same spirit of \cite[Lemma 7]{LMNN21} for Dirichlet problem.
\begin{proposition}\label{prop.DoublingDrop}
 Let $u$ and $Q$ be given as above.   Then there exists $ k_0, N_0$ such that if 
    \begin{equation*}
      N_u^\partial (Q) > N_0, \quad \text{and} \quad k > k_0, 
    \end{equation*}
    then we can find at least one $q \in \mathbb{B}_k(Q)$ such that either $N_u^\partial (q) \le \frac12 N_u^\partial (Q)$ or $Z(u)\cap q$ is empty.
\end{proposition}

A basic tool for the proof of Proposition \ref{prop.DoublingDrop} is the following version of quantitative Cauchy uniqueness proved in \cite{ZZ23}; also see earlier versions in \cite{LMNN21,T21}.

\begin{lemma}\cite[Lemma 4.1]{ZZ23}
\label{Zhuge}
Let $\Omega$ be a Lipschitz domain and $0\in\partial\Omega.$ There exists $0<\tau<1$ such that if $u$ is an $A$-harmonic function in $B_1 \cap \Omega$ satisfying $\norm{u}_{L^2(B_1 \cap \Omega)} = 1$ and $|u|+|n\cdot A\nabla u|\le \epsilon\le 1$ on $B_1\cap\partial \Omega,$ then
$$ \norm{u}_{L^\infty(B_{1/2} \cap \Omega)} \le C \epsilon^\tau.$$
\end{lemma}
\begin{proof}[Proof of Proposition \ref{prop.DoublingDrop}] 
We prove by contradiction. Assume for any $q\in \mathbb{B}_k(Q), N_u^*(q)>\frac{1}{2}N$ and $q\cap Z(u)\ne \varnothing.$ 

Let $x_Q$ be  the center of $Q$ on $\partial\Omega \cap B_{\theta R}$. 
Put $N=N_u^\partial(Q)$ and $M=J^{\frac{1}{2}}_u(x_Q,2^{S-2}\ell_Q)
=(\int_{E(x_Q,2^{S-2}\ell_Q)\cap\Omega}\mu u^2)^{\frac{1}{2}}.$ Recall that our assumption $\Lambda \in [1,2)$  implies $B(x,\sqrt{2}r) \supset E(x,r) \supset B(x, \frac12 \sqrt{2} r)$.

 By Proposition \ref{prop.u.AlmMnty}, for any $q\in \mathbb{B}_k(Q),$ we have
\begin{equation*}
    N_u^\partial(q) \le C \sup_{x\in q\cap \partial \Omega} \sum_{i=0}^{S-1}N_u(x,2^i \ell_q) \le CS \sup_{x\in q\cap \partial \Omega} N_u(x,2^{S-1} \ell_q).
\end{equation*}
It follows that for any $q\in \mathbb{B}_k(Q),$ there exists $x_q\in q\,\cap \partial\Omega$ such that
\begin{equation*}
    CSN_u(x_q,2^{S-1} \ell_q)\ge N_u^\partial(q) \ge \frac{N}{2}.
\end{equation*}
Therefore, by Proposition \ref{prop.u.AlmMnty} again and for $k_0$ sufficiently large, we deduce that  
\begin{equation}\label{est.lQ/lq}
\begin{aligned}
    \log\frac{J_u(x_q,2^{S-3}\ell_Q)}{J_u(x_q,2^{S-1}\ell_q)}&\ge \sum_{i=0}^{k-3}N_u(x_q,2^{i+S-1}\ell_q)\\
    &\ge C^{-1}(k-2)N_u(x_q,2^{S-1} \ell_q)\\
    &\ge \frac{(k-2)N}{2CS}\ge c_1kN,
    \end{aligned}
\end{equation}
for some constant $c_1 > 0$.

Since $J_u(x_q, 2^{S-3}\ell_Q) \le J_u(x_Q, 2^{S-2}\ell_Q)  \le M^2$, we obtain from \eqref{est.lQ/lq} that
\begin{equation*}
    J_u(x_q, 2^{S-1}\ell_q) \le e^{-c_1 kN} M^2.
\end{equation*}
Hence, by the elliptic regularity for the Neumman boundary value problem, we obtain from the last inequality that
\begin{equation*}
    \sup_{q \cap \partial \Omega} |u|  \le \frac{C}{\ell_q} J_u(x_q, 2^{S-1}\ell_q)^{1/2} \le C\ell_q^{-1} M e^{-c_1kN/2}.
\end{equation*}
As this estimate holds for every $q \in \mathbb{B}_k(Q)$, we then have
\begin{equation*}
    \sup_{Q\cap \partial \Omega} |u| \le C\ell_Q^{-1} 2^k M e^{-c_1kN/2},
\end{equation*}
where we also used the fact $\ell_Q=2^{k}\ell_q.$
 
Now, applying Lemma \ref{Zhuge} to $u$ with vanishing Neumann boundary conditions in $B_{\frac12 \ell_Q}(x_Q)$, 
we can obtain that
\begin{equation*}
    \sup_{E(x_Q,\frac{1}{4}\ell_Q)\cap\Omega}|u|\le CM \ell_Q^{-1}2^{k\tau}e^{-\tau c_1 kN/2},
\end{equation*}
which yields
\begin{equation}\label{est.Ju-upperbd}
J_u(x_Q,\frac{1}{4}\ell_Q)\le CM^2 2^{2k\tau}e^{-\tau c_1 kN}.
\end{equation}

On the other hand, by the definition of the boundary maximal doubling index of $Q$, and taking $r=\frac{1}{4}\ell_Q$, we get 
\begin{equation*}
\log \frac{J_u(x_Q,2^{S-2}\ell_Q)}{J_u(x_Q, \frac14 \ell_Q)} = \log \frac{M^2}{J_u(x_Q,\frac{1}{4}\ell_Q)}
\le N_u^\partial (Q) = N,
\end{equation*}
which results in
\begin{equation}\label{est.lowerbd}
    J_u(x_Q, \frac14 \ell_Q) \ge M^2 e^{-N}.
\end{equation}
Combining (\ref{est.Ju-upperbd}) and (\ref{est.lowerbd}), we obtain 
 $$
 \log C + 2k\tau \log 2 - \tau c_1 k N \ge -N.
 $$
Clearly, this is a contradiction if $k\ge k_0$ and $N\ge N_0$ for $k_0$ and $N_0$ sufficiently large.
\end{proof}

The last proposition deals with the case when $N^\partial (Q)$ is sufficiently large.
The next proposition provides the absence of nodal curves in some boundary rectangles in the case of small boundary maximal doubling index. This is a modified version for Neumann problem, in the same spirit of \cite[Lemma 8]{LMNN21} for Dirichlet problem.
\begin{proposition}\label{prop.NoZeros}
    Suppose $N_u^\partial(Q)\le N$. Then there exists $k_0>0,$ depending on $N$, such that if $k\ge k_0$, there exists $q\in \mathbb{B}_k(Q)$ such that $q\cap Z(u) = \varnothing$.
\end{proposition}
\begin{proof}
Let $x_Q$ be  the center of $Q$ on $\partial\Omega$.
Put  $M=J^{1/2}_u(x_Q,2^{S-3}\ell_Q).$   Taking $r=\frac{1}{32}\ell_Q$ in (\ref{defmaximal}), we have 
\begin{equation}\label{est.Qlower}
    \int_{B(x_Q,\frac{1}{4}\ell_Q)\cap\Omega}u^2\ge2^{-1}J_u(x_Q,2^{-3}\ell_Q)\ge 2^{-1} M^2e^{-N}.
\end{equation}
Again, applying Lemma \ref{Zhuge} in $B(x_Q,\frac{1}{2}\ell_Q)\cap\Omega$, we obtain that 
\begin{equation}\label{est.Qupper}
    \sup_{B(x_Q,\frac{1}{4}\ell_Q)\cap\Omega } |u| \le C \Big( \sup_{B(x_Q,\frac{1}{2}\ell_Q)\cap\partial\Omega} |u| \Big)^{\tau} M^{1-\tau} \ell_Q^{\tau-1}
\end{equation}
Combining (\ref{est.Qlower}) and (\ref{est.Qupper}) gives
$$ M\ell_Q^{-1}e^{-\frac{N}{2\tau}}\le C\sup_{B(x_Q,\frac{1}{2}\ell_Q)\cap\partial\Omega}|u|.$$
Hence we know that there exists $y\in B(x_Q,\frac{1}{2}\ell_Q)\cap\partial\Omega $ such that
\begin{equation}\label{positiveu}
   |u(y)|\ge \frac{1}{C}M\ell_Q^{-1}e^{-\frac{N}{2\tau}}. 
\end{equation}
 
On the other hand, the $C^{1,\alpha}$ regularity of the   Neumman boundary value problem implies the boundary gradient estimate of $u$, i.e.,
\begin{equation}\label{positiveu2}
    \norm{\nabla u}_{L^{\infty}(B(x_Q,\ell_Q)\cap\Omega )}\le CM\ell_Q^{-2}.
\end{equation}

Let $\theta =  \frac{1}{C^2}\ell_Q e^{-\frac{N}{2\tau}}$. Then (\ref{positiveu}) and (\ref{positiveu2}) combined imply $|u|>0$ in $B_\theta (y)\cap\Omega$ by the fundamental theorem of calculus. Note that $B_\theta (y)\cap\Omega$ contains at least one $q \in \mathbb{B}_k$ if $k \ge k_0$ with $k_0$ satisfying $2^{-k_0} \le ce^{-\frac{N}{2\tau}}$ for sufficiently small $c>0$ (hence $k_0 \simeq N$). This finishes the proof.
\end{proof}

\subsection{Total length of level curves}
Next, we derive 
the sharp interior estimate of nodal curves based on \cite{Ale88}.
\begin{lemma}\label{interior}
    If $q$ is a square and $u$ is $A$-harmonic in $4\Lambda q,$ then 
    \begin{equation*}
     \mathcal{H}^{1}(Z(u)\cap q)\le CN_u^{*}(q) \ell_q.  
    \end{equation*}
\end{lemma}
\begin{proof}
    This estimate is a corollary of \cite[Theorem 1]{Ale88}. In order to apply the theorem, we convert the equation into a form satisfying the desired assumptions. Since $A$ is Lipschitz, we can write the equation for $u$ in the nondivergence form as
    \begin{equation*}
        \text{tr}(A \nabla^2 u) + \text{div}(A) \cdot \nabla u = 0.
    \end{equation*}
    Dividing the equation by $\sqrt{\det(A)}$ and letting $A_1 = A/\sqrt{\det(A)}$, we get
    \begin{equation*}
        \text{tr}(A_1 \nabla^2 u) + b \cdot \nabla u = 0,
    \end{equation*}
    where $b = \text{div}(A)/\sqrt{\det (A)}$. Now clearly $\det(A_1) = 1$ and $b\in L^\infty(4\Lambda q)$.
  
    Hence, by \cite[Theorem 1]{Ale88},
    \begin{equation}\label{eq-NodalLine}
        \mathcal{H}^1(Z(u) \cap q) \le C \ell_q \Big( 1 + \log \frac{\| \nabla u \|_{L^\infty(\frac74 q)}}{ \| \nabla u \|_{L^\infty(\frac54 q)} } \Big).
    \end{equation}
    Without loss of generality, let us assume $Z(u) \cap q$ is not empty, otherwise there is nothing to prove. This means that there exists $x_0 \in q$ such that $u(x_0) = 0$.
    By the elliptic regularity estimate, we have
    \begin{equation}\label{eq-regu}
        \| \nabla u \|_{L^\infty(\frac74 q)} \le C \ell_q \bigg( \fint_{2q} u^2 \bigg)^{1/2}.
    \end{equation}
    Using the fact $u(x_0) = 0$, the elliptic regularity and Poincar\'{e} inequality, we have
    \begin{equation}\label{eq-Caccio}
        \bigg( \fint_{q} u^2 \bigg)^{1/2} \le C\ell_q^{-1} \bigg( \fint_{\frac54 q} |\nabla u|^2 \bigg)^{1/2} \le C\ell_q^{-1} \| \nabla u \|_{L^\infty(\frac54 q)}.
    \end{equation}
    Inserting \eqref{eq-regu} and \eqref{eq-Caccio} into \eqref{eq-NodalLine}, we obtain
    \begin{equation*}
        \mathcal{H}^1(Z(u) \cap q) \le C \ell_q \Big( 1 + \log \frac{\|  u \|_{L^2(2q)}}{ \| u \|_{L^2( q)} } \Big) \le C\ell_q (1 + N_u^*(q)) \le C\ell_q N_u^*(q),
    \end{equation*}
    in view of the definition of $N_u^*(q)$.
\end{proof}

The following lemma bounds the interior maximal doubling index by the boundary maximal doubling index.
\begin{lemma}\label{lem.Bk-Ik}
    Let $Q$ be a boundary rectangle decomposed into $\mathbb{B}_k(Q)$ and $\mathbb{I}_k(Q)$ as before. Then for any $q\in \mathbb{I}_k(Q)$,
    \begin{equation*}
        N_u^* (q) \le C N_u^\partial (Q).
    \end{equation*}
\end{lemma}
\begin{proof}
    This is a corollary of Lemma \ref{lem.bdry-to-interior}.
\end{proof}

In the following proposition, we estimate the total length of nodal curves in  boundary rectangles by the maximal doubling index.
\begin{proposition}\label{exterior2}
    Let $Q$ be a boundary rectangle such that $Q\subset B_{\theta R}\cap \Omega$.  Then there exists $C_0$ such that
    \begin{equation}\label{exterior1}
        \mathcal{H}^1(Z(u)\cap Q)\le C_0 N_u^{\partial}(Q)\ell_Q.
    \end{equation}
\end{proposition}
\begin{proof}
Let $N_0$ and $k_0$ be given by Proposition \ref{prop.DoublingDrop}, which depend only on $\Omega$ and $A$. Let $N=N_0$ in Proposition \ref{prop.NoZeros} and pick $k$ to be the larger $k_0$ in Proposition \ref{prop.DoublingDrop} and Proposition \ref{prop.NoZeros}.

Let $W$ be a compact subset of $\Omega.$ We would like to show
\begin{equation}\label{NodalInW}
     \mathcal{H}^1(Z(u)\cap Q\cap W)\le C_0 N_u^{\partial}(Q)\ell_Q,
\end{equation}
 for some constant $C_0$  independent of $W$. Definitely if $Q$ is small enough such that $Q\cap W=\varnothing$, then $\mathcal{H}^1(Z(u)\cap Q\cap W)=0$.
 We prove (\ref{NodalInW}) by induction from small boundary rectangles to large boundary rectangles.

 Let $(Q,\mathbb{B}_k=\mathbb{B}_k(Q),\mathbb{I}_k=\mathbb{I}_k(Q) )$ be a standard decomposition. Assume for each small boundary rectangle $q\in \mathbb{B}_k,$
 $$ \mathcal{H}^1(Z(u)\cap q\cap W)\le C_0 N_u^{\partial}(q) \ell_q. $$ Note that the base case of induction $q\cap W=\varnothing$ is trivial.

Now consider 
$$  \mathcal{H}^1(Z(u)\cap Q\cap W)\le\sum_{q\in\mathbb{I}_k}\mathcal{H}^1(Z(u)\cap q)+\sum_{q\in\mathbb{B}_k}\mathcal{H}^1(Z(u)\cap q\cap W).$$
By the interior results in Lemma \ref{interior} and Lemma \ref{lem.Bk-Ik},
\begin{equation}\label{interior2}
\begin{aligned} \sum_{q\in\mathbb{I}_k}\mathcal{H}^1(Z(u)\cap q) & \le C \sum_{q\in\mathbb{I}_k}N_u^ {*}(q) \ell_q \le C \sum_{q\in\mathbb{I}_k} C N_u^\partial (Q) \ell_q \\
& \le C 2^{k} N_u^\partial (Q) \ell_Q.
\end{aligned}
\end{equation}

For the remaining boundary rectangles $q\in \mathbb{B}_k,$ we have $N_u^{\partial}(q)\le N_u^{\partial}(Q).$ Combining Proposition \ref{prop.DoublingDrop} and Proposition \ref{prop.NoZeros}( with $N=N_0$), there is  a rectangle $q_0\in \mathbb{B}_k$ such that either $N_u^{\partial}(q_0)\le \frac{1}{2} N_u^{\partial}(Q)$ or $Z(u)\cap q_0=\varnothing.$ 
Recall that $ \mathbb{B}_k$ has $2^k$ boundary subrectangles. Applying the induction argument to each boundary subrectangles, we have 
\begin{equation}\label{exterior}
    \begin{aligned}
     &\sum_{q\in\mathbb{B}_k}\mathcal{H}^1(Z(u)\cap q\cap W)\\
     &=\sum_{q\in\mathbb{B}_k,q\ne q_0}\mathcal{H}^1(Z(u)\cap q\cap W)+\mathcal{H}^1(Z(u)\cap q_0\cap W) \\
     &\le \sum_{q\in\mathbb{B}_k,q\ne q_0}C_0 N_u^{\partial}(q) \ell_q +C_0 N_u^{\partial}(q_0)\ell_{q_0}\\
     &\le \sum_{q\in\mathbb{B}_k,q\ne q_0}C_0 N_u^{\partial}(Q) \ell_q +C_0 \Big( \frac{1}{2} \Big) N_u^{\partial}(Q)\ell_{q_0}\\
     &\le  \left \{  \frac{2^k-1}{2^k}+\frac{(1/2)}{2^k}\right \}C_0 N_u^{\partial}(Q) \ell_Q.
    \end{aligned}
\end{equation}
Now we choose $C_0$ large enough such that 
 $$ C2 ^{k} +  \left \{  \frac{2^k-1}{2^k}+\frac{1/2}{2^k}\right \} C_0\le C_0. $$
Note that $C_0$ does not depend on the compact set $W$. Thus, the combination of (\ref{interior2}) and (\ref{exterior}) gives (\ref{NodalInW}), which implies (\ref{exterior1}) by letting $W$ exhaust $Q\cap\Omega.$
\end{proof}
\begin{proof}[ Proof of  Theorem \ref{thm.LevelSet}  ]
    Theorem \ref{thm.LevelSet}  follows directly from Proposition \ref{exterior2} . Let $Q$ be a boundary rectangle centered on $0$ with $\ell_Q=\theta R.$ Thus 
    \begin{equation*}
         \mathcal{H}^1(Z_t(u) \cap B_{\theta R}) \le  \mathcal{H}^1(Z_t(u) \cap Q)\le  C_0 N_{u-t}^{\partial}(Q)\ell_Q\le CR \log \frac{\| u - t\|_{L^2(B_R \cap \Omega)}}{ \| u -t \|_{L^2(B_{\theta R} \cap \Omega )}},
    \end{equation*}
    where we used Theorem \ref{thm.Doubling} for the last inequality.
\end{proof}

\subsection{The number of boundary level points}
The following is the estimate of the number of level points on $\partial \Omega$, or equivalently the number of nodal curves which touch the boundary for Neumann problem. In the case of analytic boundaries, similar results has been studied in \cite{TZ09} and \cite{Zhu23} with different approaches.
\begin{theorem}\label{thm.levelpoints}
    Under the assumption of Theorem \ref{thm.LevelSet}, for any $t\in \R$, we have
    \begin{equation*}
        \# (\{ x\in \partial \Omega \cap B_{\theta R/2}: u(x) = t \}) \le C \log \frac{\| u - t\|_{L^2(B_R \cap \Omega)}}{ \| u -t \|_{L^2(B_{\theta R} \cap \Omega )}}. 
    \end{equation*}
\end{theorem}

\begin{proof}
    First of all, by Theorem \ref{thm.LevelSet}, we have the estimate of $\mathcal{H}^1(Z_t(u) \cap B_{\theta R})$, which is equal to the total length of the level curves contained in $B_{\theta R} \cap \Omega$. Now for each point $y \in \partial \Omega \cap B_{\theta R/2}$ with $u(y) = t$, there exists a level curve $\gamma_y = \{ x\in \overline{\Omega} \cap B_R: u(x) = u(y) \}$ connected to $y$. Moreover, we claim that this level curve cannot be entirely contained in $\Omega \cap B_{\theta R}$. In fact, if this is the case, by the maximal principle, the other endpoint of $\gamma_y$, called $y'$, is also on $\partial \Omega \cap B_{\theta R}$. Let $\mathcal{I}(y,y')$ be the portion of $\partial \Omega$ that connects $y$ and $y'$.
    Then we consider the region $D_y$ enclosed by $\gamma_y$ and $\mathcal{I}(y,y')$ and thus $\partial D_y = \gamma_y \cup \mathcal{I}(y,y')$. Observe that $u$ is $A$-harmonic in $D_y$, and satisfies the vanishing Neumann boundary condition on $\mathcal{I}(y,y')$ and constant Dirichlet boundary condition on $\gamma_y$ (i.e., $u(x) = t$ for all $x \in \gamma_y$ by our definition of level curve). This implies that $u$ is constant in $D_\gamma$ (and thus $u$ is constant in the entire $\Omega \cap B_R$), which is a contradiction. As a consequence, $\gamma_y$ must extend out of $\Omega \cap B_{\theta R}$ across the surface $\Omega \cap \partial B_{\theta R}$. But since $y \in B_{\theta R/2}$, then $\mathcal{H}^1(\gamma_y) \ge \theta R/2$. Finally, suppose that we have $m$ different points $\{ y_j: j=1,\cdots,m \}$ on $\partial \Omega \cap B_{\theta R/2}$ such that $u(y_i) = t$. Note that all the level curves $\{ \gamma_{y_j} \}$ are mutually disjoint. Hence
    \begin{equation*}
        m\theta R/2 \le \mathcal{H}^1(\cup_{j=1}^m \gamma_{y_j}) \le \mathcal{H}^1(Z_t(u) \cap B_{\theta R} ) \le C R  \log \frac{\| u - t\|_{L^2(B_R \cap \Omega)}}{ \| u -t \|_{L^2(B_{\theta R} \cap \Omega )}},
    \end{equation*}
    where we have used Theorem \ref{thm.LevelSet} in the last inequality. The desired estimate follows.
\end{proof}

The following is a version of Lin's conjecture for the Neumann problem in 2D, which is an easy corollary of Theorem \ref{thm.levelpoints}.
\begin{corollary}\label{Coro.CauchyUnique}
    Let $A$ and $\Omega$ satisfy the same assumption as Theorem \ref{thm.Doubling}. Let $u$ be $A$-harmonic in $\Omega \cap B_R$ and $\frac{\partial u}{\partial \nu} = 0$ on $\partial \Omega \cap B_R$. If for some $p \in \R$, 
    $$\# (\{ x\in \partial \Omega \cap B_{R/2}: u(x) = p \}) = \infty,$$
    then $u \equiv p$ in $\Omega \cap B_R$.
\end{corollary}

\section{The case \texorpdfstring{$\det A$}{det A} is constant}\label{sec:detA=1}
In this section, we prove Theorem \ref{thm.WithStructure}. Recall the complex analysis tool we introduced in Section \ref{subsec:quasiconformal}. If $u$ is $A$-harmonic in $\Omega$, we can find another real-valued function $v$ such that $f = u + iv$ satisfies the complex Beltrami equation \eqref{eq-dzf}. Under the assumptions of Theorem \ref{thm.WithStructure}, i.e., $A$ is symmetric and $\det A = 1$ (after divided by a constant), we have $\nu(z) = 0$ in \eqref{def.mu+nu}. Thus, \eqref{eq-dzf} is reduced to
\begin{equation*}
    \partial_{\bar{z}} f = \mu \partial_z f.
\end{equation*}
It is crucial to note that in this case, by the definition of $\mu$ in \eqref{def.mu+nu} and the assumption \eqref{A-Holder}, $\mu$ is $C^\alpha$ continuous up to the boundary by continuity.By \cite[Theorem 5.2.3]{AIM09}\footnote{The theorem is stated with $\mu \in C_0^\infty(\C)$. However, we can extend our $\mu \in C^{\alpha}(\overline{\Omega})$ to $C_0^\alpha(\C)$ and apply the Schauder estimate of elliptic system and an approximation argument to construct $\Psi \in C^{1,\alpha}(\C)$ and $\det(\nabla \Psi) \ge c$ globally. Finally, we take $\Psi$ restricted on $\overline{\Omega}$. }, we can construct a principal solution to 
\begin{equation*}
\partial_{\bar{z}} \Psi = \mu \partial_z \Psi,
  \end{equation*}
such that $\Psi:\Omega \to \Omega'$ is a $C^{1,\alpha}$ diffeomorphism and $|\nabla \Psi| \ge c > 0$ for some $c$ depending on $A$ and $\Omega$.
Hence, by Stoilow Factorization theorem \cite[Theorem 5.5.1]{AIM09}, there exits a holomorphic function $w$ such that $f(z) = w \circ \Psi (z)$. Let $h$ be the real part of $w$ (hence $h$ is harmonic), we have
\begin{equation*}
    u = h\circ \Psi \quad \text{and} \quad h = u\circ \Psi^{-1}.
\end{equation*}
Moreover, if $u$ satisfies the Neumann boundary condition $A\nabla u \cdot n = 0$ on $\Gamma \subset \partial \Omega$, then $h$ also satisfies the Neumann boundary condition $\nabla h\cdot n = 0$ on $\Gamma' = \Psi(\Gamma) \subset \partial \Omega'$.

Since $\Psi: \Omega \to \Omega'$ is bi-Lipschitz, 
then the doubling property and the length of level curves of $u$ are equivalent to those of harmonic function $h$ in $\Omega'$.

\begin{proof}[Proof of Theorem \ref{thm.WithStructure}]
By the previous reduction, it is sufficient to consider harmonic function $h$ in $\Omega' = \Psi(\Omega)$ with vanishing Neumann boundary condition. Recall that $\Psi: \Omega \to \Omega'$ is a $C^{1,\alpha}$ diffeomorphism and $|\nabla \Psi| \simeq 1$. If $\Omega$ is $C^{1,\alpha}$, then $\Omega'$ is also $C^{1,\alpha}$, then the desired estimates follows from Theorem \ref{thm.Doubling} and Theorem \ref{thm.LevelSet}. Now if $\Omega$ is convex, then $\Omega'$ may be neither $C^{1,\alpha}$ nor convex. However, it belongs to a class of quasiconvex domain introduced in \cite[Definition 1.1]{ZZ23} with continuity modulus $\omega(r) \lesssim r^\alpha$. 

By \cite[Remark 3.6]{ZZ23}, as before, we obtain the almost monotonicity of the stream function of $h$ in $\Omega'$. By this and repeat the exactly the same argument in Section \ref{sec:DoublingIndex} and Section \ref{sec:levelset}, we obtain the desired estimates. The details are omitted.
\end{proof}

\appendix
\section{A universal lower bound of doubling index}

We show that if $u$ is a nontrivial $A$-harmonic function in $\Omega \cap B_R$ satisfying the Neumann boundary condition on $\partial \Omega \cap B_R$ and $0 \in \partial \Omega$, then $N_u(0,r) \ge c$ and $N_{\nabla u}(0,r) \ge c$ for some universal constant $c>0$ independent of $u$. 
\begin{proposition}\label{prop.lowbd.doubling}
    Let $\Omega$ be $C^{1,\alpha}$ and $0\in \partial \Omega$. Let $u$ be $A$-harmonic in $\Omega \cap B_R$ subject to the vanishing Neumann boundary condition on $\partial \Omega \cap B_R$. Then there exists $\delta \in (0,1)$ depending only on $A$ and $\Omega$ such that for any $r\in (0, R/2)$,
    \begin{equation}\label{est.Jur<2r}
        J(0,r) \le (1- \delta) J(0,2r).
    \end{equation}
    In particular, we have $N_u(0,r) \ge c$ for some constant $c>0$ (depending only on $A$ and $\Omega$).
\end{proposition}
\begin{proof}
    It suffices to prove \eqref{est.Jur<2r}.
    Without loss of generality, assume $A(0) = I$. We first prove
    \begin{equation}\label{est.Br<B2r}
        \int_{B_r \cap \Omega} u^2 \le C \int_{( B_{2r} \setminus B_r) \cap \Omega} u^2.
    \end{equation}

    By the maximal principle, the maximum of $u$ in $B_{3r/2}$ is attained on $\partial (B_{3r/2} \cap \Omega)$. But due to the Hopf's lemma over $C^{1,\alpha}$ boundary, the local maximum value cannot attained on $\partial \Omega$. Thus
    \begin{equation*}
        \| u \|_{L^\infty(B_{3r/2} \cap \Omega)} = \| u \|_{L^\infty(\partial B_{3r/2} \cap \Omega)}.
    \end{equation*}
    It follows that
    \begin{equation*}
        \int_{B_r \cap \Omega} u^2 \le Cr^d\| u \|_{L^\infty(B_{3r/2} \cap \Omega)}^2 \le Cr^d \| u \|_{L^\infty(\partial B_{3r/2} \cap \Omega)}^2 \le C \int_{( B_{2r} \setminus B_r) \cap \Omega} u^2,
    \end{equation*}
    where in the last inequality, we have used the local $L^\infty$ estimate for the Neumann problem. This proves \eqref{est.Br<B2r}.

  Note that \eqref{est.Br<B2r} yields
    \begin{equation*}
        \int_{B_r \cap \Omega} \mu u^2 \le C \int_{(B_{2r} \setminus B_r) \cap \Omega} \mu u^2. 
    \end{equation*}
    Finally, by a hole-filling argument, we have
    \begin{equation*}
        J(0,r) \le \frac{C}{1+C} J(0,2r) = (1-\delta) J(0,2r),
    \end{equation*}
    as desired.
\end{proof}

\begin{proposition}\label{prop.lowbd.NDu}
    Let $\Omega$ and $u$ be the same as Proposition \ref{prop.lowbd.doubling}. Then there exists $\delta \in (0,1)$ such that for $r\in (0,R/2)$,
    \begin{equation}\label{est.Dur<2r}
        \| \nabla u \|_{L^2(B_r \cap \Omega)} \le (1 - \delta) \| \nabla u \|_{L^2(B_{2r} \cap \Omega)}.
    \end{equation}
    In particular, $N_{\nabla u}(0,r) \ge c$ for some constant $c>0$ (depending only on $A$ and $\Omega$).
\end{proposition}

\begin{proof}
    Let $\ell = \fint_{\partial B_{2r} \cap \Omega}$. Then $u - \ell$ is also a solution. By the Caccioppoli inequality, \eqref{est.Br<B2r} and the Poincar\'{e} inequality, we have
    \begin{equation*}
        \int_{B_r\cap \Omega} |\nabla u|^2 \le Cr^{-2} \int_{B_{2r}\cap \Omega} |u - \ell |^2 \le Cr^{-2} \int_{(B_{2r}\setminus B_r)\cap \Omega} |u - \ell |^2 \le C\int_{(B_{2r}\setminus B_r)\cap \Omega} |\nabla u|^2.
    \end{equation*}
    By a hole-filling argument as before, we obtain \eqref{est.Dur<2r}.
\end{proof}
\bibliographystyle{abbrv}
\bibliography{reference-2}

\end{document}